%% file: main-arXiv.tex
\documentclass[final,3p,times]{elsarticle}

\journal{arXiv}

\input{packagesAndMacros}

\begin{document}

\begin{frontmatter}

\title{Enhanced Relaxed Physical Factorization preconditioner for coupled poromechanics}

\author[label1]{Matteo Frigo}
\ead{matteo.frigo.3@phd.unipd.it}

\author[label2]{Nicola Castelletto}
\ead{castelletto1@llnl.gov}

\author[label1]{Massimiliano Ferronato}
\ead{massimiliano.ferronato@unipd.it}

\address[label1]{Department ICEA, University of Padova, Padova, IT}
\address[label2]{Atmospheric, Earth and Energy Division, Lawrence Livermore National Laboratory, United States}

\begin{abstract}
{
  The {relaxed physical factorization} (RPF) preconditioner 
  %
  is a recent algorithm allowing for the efficient and robust solution to the block linear systems arising from the three-field displacement-velocity-pressure formulation of coupled poromechanics.
  For its application, however, 
  it is necessary to invert blocks with the algebraic form $\hat{C} = ( C + \beta F F^T)$, where $C$ is a symmetric positive definite matrix, $FF^T$ a rank-deficient term, and $\beta$ a real non-negative coefficient. 
  The inversion of $\hat{C}$, performed in an inexact way, can become unstable for large values of $\beta$, as it usually occurs at some stages of a full poromechanical simulation.
  In this work, we propose a family of {algebraic techniques} to stabilize the inexact solve with $\hat{C}$. This strategy can prove useful in other problems as well where such an issue might arise, such as augmented Lagrangian preconditioning techniques for Navier-Stokes or incompressible elasticity.
  First, we introduce an iterative scheme obtained by a natural splitting of matrix 
  $\hat{C}$.
  Second, we develop a technique based on the use of a proper projection operator 
  annihilating the near-kernel modes of $\hat{C}$.
  Both approaches give rise to a novel class of preconditioners denoted as Enhanced RPF (ERPF).
  Effectiveness and robustness of the proposed algorithms are demonstrated in both theoretical benchmarks and real-world large-size applications, outperforming the native RPF preconditioner.
  }
\end{abstract}


\end{frontmatter}

\section{Introduction}\label{sec:int}
Numerical simulation of Darcy's flow in a porous medium coupled with quasi-static mechanical deformation is based on poromechanics theory \cite{Bio41,Cou04}.
{We consider saturated single-phase flow of a slightly compressible fluid through a deformable medium.
The set of governing equations for the displacement-velocity-pressure formulation of the poroelastic problem, e.g., \cite{Fri00,PhiWhe07a,PhiWhe07b,BauRadKoc17,Hu_etal17,Bot_etal18} reads: }
\begin{linenomath}
{\begin{subequations}
\begin{align}
  \nabla \cdot \left( \tensorFour{C}_{dr} : \nabla^{s} \tensorOne{u} - b p \tensorTwo{1} \right)
  &=
  \tensorOne{0} & &\mbox{ in } \Omega \times \mathcal{I}
  & &\mbox{(equilibrium)}, \label{momentumBalanceS}\\
  \mu \tensorTwo{\kappa}^{-1} \cdot \tensorOne{q} + \nabla p
  &=
  \tensorOne{0} & &\mbox{ in } \Omega \times \mathcal{I}
  & &\mbox{(Darcy's law)}, \label{darcyS}  \\
  b \nabla \cdot \dot{\tensorOne{u}} + S_{\epsilon} \dot{p} + \nabla \cdot \tensorOne{q} &=
  f & &\mbox{ in } \Omega \times \mathcal{I}
  & &\mbox{(continuity)}, \label{massBalanceS}
\end{align}\label{eq:IBVP}\null 
\end{subequations}}
\end{linenomath}
{with $\Omega\subset\mathbb{R}^3$ a bounded domain and $\mathcal{I}=(0,t_{max})$  an open time interval. 
In \eqref{eq:IBVP}, $\tensorFour{C}_{dr}$ is the rank-four elasticity tensor, $\nabla^{s}$ is the symmetric gradient operator, 
$b$ is the Biot coefficient, and $\tensorTwo{1}$ is the rank-two        
identity tensor; $\mu$ and $\tensorTwo{\kappa}$ are the fluid viscosity and the rank-two permeability tensor, respectively; 
$S_{\epsilon}$ is the constrained specific storage coefficient, i.e. the reciprocal of Biot's modulus, and $f$ the fluid source term.
The primary unknowns are denoted by $\tensorOne{u}$ (displacement), $\tensorOne{q}$ (Darcy's velocity), and $p$ (excess pore pressure).}
%
%
The problem is completed with proper boundary conditions:
\begin{linenomath}
\begin{subequations}
\begin{align}
  \tensorOne{u} &= \bar{\tensorOne{u}}
  & &\mbox{ on } \Gamma_u \times \mathcal{I}, & &\mbox{(prescribed displacement)}, \label{momentumBalanceS_DIR}\\
  \left( \tensorFour{C}_{dr} : \nabla^{s} \tensorOne{u} - b p \tensorTwo{1} \right) \cdot \tensorOne{n} &= 
  \bar{\tensorOne{t}}
  & &\mbox{ on } \Gamma_{\sigma} \times \mathcal{I}, & &\mbox{(prescribed traction)}, \label{momentumBalanceS_NEU}\\
  \tensorOne{q} \cdot \tensorOne{n} &= \bar{q}
  & &\mbox{ on } \Gamma_{q} \times \mathcal{I}, & &\mbox{(prescribed discharge)}, \label{massBalanceS_NEU}\\    
  p &=\bar{p}
  & &\mbox{ on } \Gamma_p \times \mathcal{I}, & &\mbox{(prescribed pressure)}, \label{massBalanceS_DIR}
\end{align}\label{eq:IBVP_b}\null
\end{subequations}
\end{linenomath}
where $\bar{\tensorOne{u}}$, $\bar{\tensorOne{t}}$, $\bar{q}$, and $\bar{p}$ are the boundary displacement, traction, Darcy velocity and excess pore pressure, respectively, on the portions $\Gamma_u$, $\Gamma_{\sigma}$, $\Gamma_q$, and $\Gamma_p$ of the frontier $\partial \Omega$ such that $\Gamma_u\cup\Gamma_{\sigma}=\Gamma_q\cup\Gamma_p=\partial\Omega$ and $\Gamma_u\cap\Gamma_{\sigma}=\Gamma_q\cap\Gamma_p=\emptyset$.
%
The initial condition can be given as
\begin{equation}
  b \nabla \cdot \tensorOne{u} (\tensorOne{x}, 0) + S_{\epsilon} p (\tensorOne{x}, 0) = b \nabla \cdot {\tensorOne{u}_0 } + S_{\epsilon} {p_0},
  \qquad
  \tensorOne{x} \in \overline{\Omega},
  \label{eq:inicon}
\end{equation}
with $\tensorOne{u}_0$ and $p_0$ the initial displacement and excess pore pressure field. While equation \eqref{eq:inicon} specifies the initial fluid content of the medium \cite{Bio41}, in practice $p_0$ is often directly measured or computed, and $\tensorOne{u}_0$ is then obtained so as to satisfy the momentum balance \eqref{momentumBalanceS}---see, e.g. \cite{GirKumWhe16,GirWheAlmDan19}.
%
The focus of this work is the iterative solution of the linear algebraic system arising from the discretization of the initial boundary value problem \eqref{eq:IBVP}-\eqref{eq:inicon}. 
%
In particular, we consider the block linear system $\blkMat{A} \vec{x} = \vec{F}$ obtained by combining a mixed finite element discretization in space with implicit integration in time using the $\theta$-method \cite{FerCasGam10}:
\begin{linenomath}
\begin{align}
	\blkMat{A} &=
	\begin{bmatrix}
		K & 0 & -Q \\
		0 & A & -B \\
		Q^T & \gamma B^T & P
	\end{bmatrix},
	&
	\vec{x} &= 
	\begin{bmatrix}
		\vec{u} \\
		\vec{q} \\
		\vec{p}
	\end{bmatrix},
	&
	\vec{F} &=
	\begin{bmatrix}
		\vec{f}_u \\
		\vec{f}_q \\
		\vec{f}_p
	\end{bmatrix},
	\label{eq:sys}
\end{align}
\end{linenomath}
where $\vec{u} \in \mathbb{R}^{n_u}$, $\vec{q} \in \mathbb{R}^{n_q}$ and $\vec{p} \in \mathbb{R}^{n_p}$ denote the vectors containing the unknown $n_u$ displacement, $n_q$ Darcy velocity, and $n_p$ pressure degrees of freedom, respectively, at discrete time $t_{n+1}$, and $\gamma=\theta \Delta t$, with $\Delta t=(t_{n+1}-t_{n})$ the time integration step size and $\theta$ a real parameter $(1/2 \le \theta \le1)$.
In \eqref{eq:sys}, $K$ is the classic small-strain structural stiffness matrix, $A$ is the (scaled) velocity mass matrix, $P$ is the (scaled) pressure mass matrix, $Q$ is the poromechanical coupling block and $B$ is the Gram matrix.
We employ $\mathbb{Q}_1$ elements for the displacement field, $\mathbb{RT}_0$ elements for the velocity field, and $\mathbb{P}_0$ elements for the pressure field. 
Assuming a linear elastic law for the mechanical behavior of the porous medium leads to a symmetric positive definite (SPD) matrix $K$.
Matrix $A$ is SPD as well, while $P$ is diagonal with non-negative entries.
Additional details and explicit expressions for matrices in $\blkMat{A}$ and vectors in the block right-hand side $\vec{F}$ can be found in \cite{FerCasGam10,CasWhiFer16}. 


{
The two-field displacement-pressure formulation discretized by the continuous Galerkin (CG) finite elements is  probably the most common technique for numerical modeling of coupled poromechanics \cite{LewSch98}.
In constrast to the CG method, the selected three-field formulation yields locally mass-conservative velocity fields and is robust for highly heterogeneous permeability fields, a key requirement in several geoscience applications, e.g., \cite{CasFerGam12,JhaJua14}.
We note that an enriched Galerkin (EG) method can be also used to fix this drawback of the CG method \cite{ChoLee18,Cho18}.
If the selected combination of discretization spaces does not intrinsically satisfy the inf-sup stability in the undrained limit \cite{Lip02}, 
%
proper stabilization strategies can be introduced to eliminate the spurious oscillation modes arising in the pressure solution. 
These techniques can be subdivided into two classes, based on whether they: 
(i) enrich the discrete variable spaces so that the Ladyzhenskaya-Babu\v{s}ka-Brezzi (LBB) condition is met, e.g.~\cite{Rod_etal18,NiuRuiHu19}, or 
(ii) introduce a stabilization term to guarantee the solvability of the resulting saddle-point problem, e.g.~\cite{HonMalFerJan18,CamWhiBor19}. 
The first strategy is mathematically elegant and robust, 
but can modify the algebraic structure of problem \eqref{eq:sys}. 
The selected $\mathbb{Q}_1$-$\mathbb{RT}_0$-$\mathbb{P}_0$ mixed discretization can be effectively stabilized following the second strategy by adding a proper term to the mass balance equation~\cite{FriCasFerWhi20}.
This approach has the advantage of having a minor impact on the algebraic structure of the problem \eqref{eq:sys}, 
so that the solution algorithms developed in this work can be used independently of the presence of such a contribution.}

Our goal is the efficient iterative solution of the linear system \eqref{eq:sys} with the aid of preconditioned Krylov subspace methods.
Since $\blkMat{A}$ is non-symmetric, or indefinite if properly symmetrized, a global Bi-CGStab \cite{Vor92} or GMRES \cite{SaaSch86} algorithm can be used as a solver.
%
%
{
Over the past decade, a growing interest has regarded the implicit solution 
of coupled poromechanics, with the introduction of a number of methods built on both fully-algebraic approaches and discretization-based strategies.
%
The proposed algorithms can be grouped into two main categories:
(i) \textit{sequential-implicit} methods, in which the primary variables are updated one at a time, iterating between governing equations 
\cite{JhaJua07,GirKumWhe16,Alm_etal16,Bot_etal17,Bor_etal18,DanGanWhe18,DanWhe18,Hon_etal18}; and (ii) \textit{fully-coupled} approaches, which solve the global system of equations simultaneously 
for all the primary unknowns, with an emphasis on
preconditioned Krylov solvers with scalable and parameter-robust 
methods \cite{Lip02,Kuz_etal03,FerCasGam10,WhiBor11,HagOsnLan12a,HagOsnLan12b,AxeBlaByc12,TurArb14,CasWhiFer16,GasRod17,Luo_etal17,LeeMarWin17,Adl_etal18,HonKra18,FerFraJanCasTch19,Adl_etal20,Bui_etal20}. 
Sequential-implicit methods exhibit a linear convergence, but can take advantage of using distinct, and specialized, codes for the structural and the flow models.
In contrast,
fully-coupled approaches ensure unconditional stability with a super-linear convergence, but require the design of dedicated preconditioning operators.
This research field, in particular, has attracted an increasing interest from the scientific community, with the introduction of a number of different algorithms in the last few years.
%
Just to cite a few recent examples of preconditioners for coupled poromechanics, we mention spectrally-equivalent block diagonal and block triangular strategies \cite{Lip02,Kuz_etal03,TurArb14,LeeMarWin17,HonKra18,Adl_etal20},
multigrid-based scalable methods \cite{GasRod17,Luo_etal17,Bui_etal20},
or block preconditioners combining physically-based and algebraic tricks to construct different Schur complement 
approximations \cite{FerCasGam10,CasWhiFer16,FerFraJanCasTch19,FriCasFerWhi20}. 
A recent novel approach belonging to the category of fully-coupled methods is the relaxed physical factorization (RPF) preconditioner \cite{FriCasFer19}.
The distinctive feature of the RPF algorithm is that it does not rely on accurate sparse approximations of the Schur complements, with the convergence accelerated by setting a nearly-optimal relaxation parameter $\alpha$.
Similar approaches were originally developed for the solution of the linear system arising from the discretization of the Navier-Stokes equations \cite{BenNgNiuZhe11,BenDepGraQua16}.

The RPF technique has been shown very efficient in solving challenging problems with severe material heterogeneities {and proved competitive with other well-established block preconditioners \cite{FriCasFer19}.}
However, a performance degradation can be observed when the time-step  approaches 0 (undrained conditions) or $+\infty$ (fully drained conditions).
This behavior is due to the increase of the ill-conditioning of the inner blocks, which causes the worsening of the RPF  efficiency and robustness.
In fact, the native RPF preconditioner requires the solution of inner systems in the form $\hat{C}=(C+\beta FF^T)$, where $C$ is SPD and $FF^T$ is rank-deficient, and $\beta$ is a scalar coefficient that may tend to infinity in the limit case of $\frac{\Delta t}{t_c} \ll 1$ (undrained conditions) or $\frac{\Delta t}{t_c} \gg 1$ (uncoupled consolidation), with $t_c$ the characteristic consolidation time \cite{Cou04}.
Problematic issues with a very similar algebraic origin may frequently occur in several other important applications 
beyond coupled poromechanics and independently of the selected discretization spaces, e.g., in augmented Lagrangian approaches for the Navier-Stokes equations or mixed formulations of incompressible elasticity \cite{BriSte96,OlsBen07,BenOls11,BenOlsWan11,FarMitWec19}.
In this paper, we advance two methods to stabilize the solves with the inner blocks $\hat{C}$.
In the first approach, a natural splitting of $\hat{C}$ is proposed to obtain a convergent sequential iterative scheme.
In the second approach, a proper projection operator is introduced to annihilate the near-kernel modes of $\hat{C}$.
{We would like to emphasize that the proposed two strategies are fully algebraic, hence they are naturally amenable to a generalization to other applications where such an issue might arise.}

The paper is organized as follows.
In Section \ref{sec:RPF} a brief review of the RPF preconditioner is provided, {with a focus on the numerical challenges associated with} large and small time-step values.
Section \ref{sec:RPFissues} provides the theoretical basis for the two methods introduced for the RPF stabilization, with the related properties, performance 
and robustness investigated in Section \ref{sec:NumRes} through both academic benchmarks and real-world applications.
Finally, a few conclusive remarks close the paper in Section \ref{sec:conc}.

\section{The RPF Preconditioner }\label{sec:RPF}
In this section, we briefly revisit the RPF preconditioner \cite{FriCasFer19} recalling only the aspects necessary for the 
remainder of this work.
In factored form, the RPF preconditioner $\blkMat{M}^{-1}$ reads:
\begin{equation}\label{eq:RPF}
	\blkMat{M}^{-1}=\alpha \left(\blkMat{M}_1 \blkMat{M}_2\right)^{-1}=\alpha \left(
  \begin{bmatrix} K & 0 & -Q\\ 0 & \alpha I_q & 0\\ Q^T & 0 & \alpha I_p\\ \end{bmatrix}
  \begin{bmatrix} \alpha I_u & 0 & 0\\ 0 & A & -B\\ 
  0 &\gamma B^T & \alpha I_p\\ \end{bmatrix} \right)^{-1},
\end{equation} 
where $\alpha>\|P\|_\infty$ is the relaxation parameter, and $I_u$, $I_q$, and $I_p$ are the identity matrix in 
$\mathbb{R}^{n_u \times n_u}$, $\mathbb{R}^{n_q \times n_q}$, and $\mathbb{R}^{n_p \times n_p}$, respectively. 
Setting $\alpha$ is key to accelerate the convergence of the Krylov subspace method preconditioned with \eqref{eq:RPF} 
to solve \eqref{eq:sys}. 
Though the convergence rate of non-symmetric Krylov solvers does not depend only on the eigenvalue distribution 
of the preconditioned matrix, the computational practice suggests that a clustered eigenspectrum rarely leads to 
slow convergence. Hence, the criterion for setting $\alpha$ relies on clustering as much as possible the eigenvalues 
of the preconditioned matrix $\blkMat{T}=\blkMat{M}^{-1}\blkMat{A}$ around 1. To this aim, $\alpha$ can be selected 
such that the trace of $\blkMat{T}$ is as close as possible to the system size:
\begin{equation} \label{eq:traceA}
  \alpha = \arg \min_{\alpha \geq \|P\|_{\infty} } \left ( n_u + n_q + n_p - \operatorname{tr}  \left[ \blkMat{T} 
	\right] \right ).
\end{equation}
Following the arguments developed in \cite{FriCasFer19}, equation \eqref{eq:traceA} can be rewritten as:
\begin{equation} \label{eq:traceZ}
  \alpha = \arg \min_{\alpha \geq \|P\|_{\infty} } \left (n_p - \operatorname{tr}  \left[ Z_{\alpha} \right] \right ),
\end{equation}
where
\begin{equation}\label{eq:Z_alp}
  Z_{\alpha}= \alpha \left( \alpha I_p + S_K \right)^{-1} S  \left( \alpha I_p + \gamma S_A \right)^{-1},
\end{equation}
and
\begin{eqnarray}
  S_K &=& Q^TK^{-1}Q, \label{eq:SK} \\
  S_A &=& B^TA^{-1}B, \label{eq:SA}\\
  S_{\; \;} &=& P + S_K + \gamma S_A. \label{eq:S}
\end{eqnarray}
Finding analytically $\alpha$ as in \eqref{eq:traceZ} is not practically feasible, because the matrices \eqref{eq:SK}, 
\eqref{eq:SA}, \eqref{eq:S} are dense. However, for the sake of $\alpha$ estimate only, we can replace $S_K$ and $S_A$ with 
diagonal matrices, $D_K$ and $D_A$, respectively. As far as $S_K$ is concerned, it can be effectively approximated by 
the so-called ``fixed-stress'' matrix \cite{CasWhiTch15,WhiCasTch16}, also computed in a fully algebraic way \cite{CasWhiFer16}.
For the matrix $S_A$, we follow the approach used in \cite{BerManMan98} by defining:
\begin{equation}\label{eq:DA}
  D_A= \operatorname{diag} \left ( B^T \tilde{A}^{-1} B\right),
\end{equation}
where
\begin{equation}
\tilde{A}=\operatorname{diag}\left( a_1, a_2, \ldots, a_{n_q} \right), \qquad
a_i=\left( \sum_{j=1}^{n_q} \left| A_{ij} \right|^2 \right)^{1/2}, \qquad i\in\{1,\ldots,n_q\}.
\label{eq:tilde_A}
\end{equation}
With the diagonal approximations $D_K$ and $D_A$ instead of $S_K$ and $S_A$, the trace of $Z_{\alpha}$ can be computed 
at a negligible cost. The value of $\alpha$ is therefore obtained from \eqref{eq:traceZ} as \cite{FriCasFer19}:
\begin{equation}
	\alpha = \frac{\sqrt{\gamma}}{n_p} \sum_{i=1}^{n_p} \sqrt{D_K^{(i)} D_A^{(i)}}.
	\label{eq:alpha_opt}
\end{equation}
\begin{remark}
Throughout this paper, we overload the operator $\operatorname{diag}(\cdot)$ to create a square diagonal matrix.
In particular, if $\mybold{v} \in \mathbb{R}^n$, then $\operatorname{diag}(\mybold{v})$ returns a diagonal matrix 
with the elements of vector $\mybold{v}$ (Eq. \eqref{eq:tilde_A}) ; 
otherwise, if $V \in \mathbb{R}^{n \times n}$, then $\operatorname{diag}(V)$ returns a diagonal matrix 
consisting of the main  diagonal of $V$ (Eq. \eqref{eq:DA}).
\end{remark}
\begin{remark}
We emphasize that, in many approaches, the design of robust and efficient preconditioners based on accurate approximations of matrices \eqref{eq:SK}, 
\eqref{eq:SA}, and \eqref{eq:S}---i.e., the Schur complement approximation---is the  focus. Conversely, in the RPF framework cheap diagonal approximations of such matrices are enough to enable effective estimates for the parameter $\alpha$, the key component to the success of this preconditioning technique.
\end{remark}

To apply $\blkMat{M}^{-1}$ to a vector, it is convenient to rewrite \eqref{eq:RPF} as:
\begin{equation}\label{eq:M1M2M3M4}
  \mybold{M}^{-1}=\left(
  \begin{bmatrix} I_u & 0 & -\frac{1}{\alpha}{Q}\\ 0 & I_q & 0 \\ 0 & 0 & I_p\\ \end{bmatrix}
  \begin{bmatrix} \hat{K} & 0 & 0\\ 0 & I_q & 0 \\ Q^T & 0 & I_p\\ \end{bmatrix}
  \begin{bmatrix} I_u & 0 & 0\\ 0 & \hat{A} & -B \\ 0 & 0 & \alpha I_p\\ \end{bmatrix}
  \begin{bmatrix} I_u & 0 &0 \\ 0 & I_q & 0 \\ 0 & \frac{\gamma}{\alpha}{B^T} & I_p\\ \end{bmatrix}
  \right)^{-1},
\end{equation}
with 
\begin{eqnarray}
	\hat{K} &=& K + \frac{1}{\alpha} Q Q^T, \label{eq:Khat} \\
	\hat{A} &=& A + \frac{\gamma}{\alpha} B B^T. \label{eq:Ahat}
\end{eqnarray}
Equation \eqref{eq:M1M2M3M4} shows that the RPF application requires two inner solves with matrices $\hat{K}$ and
$\hat{A}$ of \eqref{eq:Khat}-\eqref{eq:Ahat}. 
Notice that 
such matrices depend on $\alpha$, whose inverse multiplies the rank-deficient matrices $Q Q^T$ and $B B^T$. The latter
could prevail on the SPD contributions $K$ and $A$ for relatively small values of $\alpha$, thus affecting the accuracy 
and numerical stability in the application of $\hat{K}^{-1}$ and $\hat{A}^{-1}$. 
For this reason, two limiting values, $\alpha_K$ and $\alpha_A$, are introduced, with the definitions \eqref{eq:Khat}-\eqref{eq:Ahat} modified as:
\begin{eqnarray}
	\hat{K} &=& K + \frac{1}{\max\left(\alpha,\alpha_K\right)} Q Q^T, \label{eq:Khat_mod} \\
	\hat{A} &=& A + \frac{\gamma}{\max\left(\alpha,\alpha_A\right)} B B^T. \label{eq:Ahat_mod}
\end{eqnarray}
Following the result of Theorem 3.1 in \cite{FriCasFer19}, the bounds $\alpha_K$ and $\alpha_A$ read:
\begin{equation}\label{eq:limit_alpha_ill}
  \alpha_K = \frac{\lambda_1(S_K)}{\overline{\omega_K} - 1}, \qquad
  \alpha_A = \frac{\gamma \lambda_1(S_A)}{\overline{\omega_A} - 1},
\end{equation}
where $\lambda_1(S_K)$ and $\lambda_1(S_A)$ are the spectral radii of $S_K$ and $S_A$, and $\overline{\omega_K}$ and $\overline{\omega_A}$ are user-specified parameters denoting the maximum acceptable value for the ratios 
$\kappa(\hat{K})/ \kappa({K})$ and $\kappa(\hat{A})/ \kappa({A})$, i.e., the maximum increase of the spectral
condition number of $\hat{K}$ and $\hat{A}$ with respect to that of $K$ and $A$, respectively. 
For the practical computation of $\alpha_K$ and $\alpha_A$, in equation \eqref{eq:limit_alpha_ill} $S_K$
and $S_A$ can be replaced by $D_K$ and $D_A$, respectively.


The RPF set-up and application to the vector $\mybold{r}^T=[\mybold{r}_u^T,\mybold{r}_q^T,\mybold{r}_p^T]$ are summarized 
in Algorithm \ref{alg:RPF0} and \ref{alg:appRPF}, respectively. Of course, lines 3 and 8 of Algorithm \ref{alg:appRPF}
can be replaced by an approximate application of $\hat{K}^{-1}$ and $\hat{A}^{-1}$.

\begin{algorithm}[tbh]
	\caption{RPF computation: $\blkMat{M}=\mbox{\sc RPF\_SetUp}(\gamma,\overline{\omega_K},\overline{\omega_A},D_K,D_A,\blkMat{A})$}
  \label{alg:RPF0}
  \begin{algorithmic}[1]
	  \State $\lambda_1(D_{K})=\max_i D_K^{(i)}$
    \State $\lambda_1(D_{A})=\max_i D_A^{(i)}$ 
    \State $\alpha_K = (\overline{\omega_K}-1)^{-1} \lambda_1 (D_K)$
    \State $\alpha_A = (\overline{\omega_A}-1)^{-1} \gamma \lambda_1 (D_A)$
	  \State $\alpha=\sqrt{\gamma} n_p^{-1} \sum_i \sqrt{D_K^{(i)}D_A^{(i)}}$
    \State $\hat{K}=QQ^T$  
    \State $\hat{A}=BB^T$    
    \State $\hat{K} \gets K + \max(\alpha,\alpha_K)^{-1} \hat{K}$
	  \State $\hat{A} \gets A + \gamma \max(\alpha,\alpha_A)^{-1} \hat{A}$
  \end{algorithmic}
\end{algorithm}


\begin{algorithm}
	\caption{RPF application: $[\mybold{t}_u^T,\mybold{t}_q^T,\mybold{t}_p^T]=\mbox{\sc RPF\_Apply}(\mybold{r}_u,\mybold{r}_q,\mybold{r}_p,\blkMat{A},\blkMat{M})$}
  \label{alg:appRPF}
  \begin{algorithmic}[1]
    \State $\mybold{x}_u=\alpha^{-1} Q \mybold{r}_p$
    \State $\mybold{x}_u \gets \mybold{r}_u + \mybold{x}_u$
	  \State \textbf{Solve} $\hat{K}\mybold{t}_u=\mybold{x}_u$
    \State $\mybold{y}_p=Q^T \mybold{t}_u$
    \State  $\mybold{y}_p \gets \mybold{r}_p - \mybold{y}_p$
    \State $\mybold{z}_q=B \mybold{y}_p$
    \State $\mybold{z}_q \gets \mybold{r}_q + \alpha^{-1} \mybold{z}_q$      
    \State \textbf{Solve} $\hat{A}\mybold{t}_q=\mybold{z}_q$   
    \State $\mybold{t}_p=B^T \mybold{t}_q$
    \State $\mybold{t}_p \gets \alpha^{-1} \left( \mybold{y}_p - \gamma \mybold{t}_p \right)$     
  \end{algorithmic}
\end{algorithm}

\section{Enhanced RPF preconditioner}\label{sec:RPFissues}
The use of equations \eqref{eq:Khat_mod}-\eqref{eq:Ahat_mod} instead of \eqref{eq:Khat}-\eqref{eq:Ahat} often is not sufficient to guarantee the RPF efficiency.
Especially for large time-step values, a degradation of the RPF performance can be observed. There are several reasons 
for this behavior. First, the value of $\alpha_K$ and $\alpha_A$ can be quite far from the optimal value of $\alpha$, 
in particular for $\frac{\gamma}{t_c} \gg 1$. This can be heuristically noted by considering that $\alpha$ is proportional to
$\sqrt{\gamma}$ (equation \eqref{eq:alpha_opt}), while $\alpha_K$ is constant and $\alpha_A$ varies linearly with $\gamma$
(equation \eqref{eq:limit_alpha_ill}). The separation between $\alpha_K$ and $\alpha$ is finite for $\gamma \rightarrow 0$, 
whereas for $\gamma \rightarrow \infty$ the separation between $\alpha_A$ and $\alpha$ grows to infinity. Hence, the larger 
the value of $\gamma$, the greater is the distance between the optimal $\alpha$ and $\alpha_A$.

The impact of setting a bound to $\alpha$ on the overall quality of the RPF preconditioner can be investigated as follows. 
For the sake of generality, we use the notation:
\begin{equation}\label{eq:hat_Cu_b}
  \hat{C} \mybold{w}= \mybold{b},
\end{equation}
with:
\begin{equation}\label{eq:hat_C}
  \hat{C} = C + \beta F   F^T,
\end{equation}
to denote the inner solves required by the RPF application (lines 3 and 8 of Algorithm \ref{alg:appRPF}). The following
result holds.

\begin{theorem}\label{thm:thm1}
	Let $C\in\mathbb{R}^{n\times n}$ and $F\in\mathbb{R}^{n\times n_p}$, with $n_p<n$, be an SPD and a full-rank matrix, 
	respectively. If $\beta > \beta_\ell > 0$, then the $n$ eigenvalues $\lambda_i$ of the generalized eigenproblem  
	\begin{equation}\label{eq:geneig} 
		\left ( C + \beta FF^T \right) \mybold{v}_i = \left ( C + \beta_\ell FF^T \right) \lambda_i \mybold{v}_i 
	\end{equation}
	are bounded in the interval $[1,\lambda_1]$, where: 
	\begin{equation}\label{eq:lambda1}
		\lambda_1 = \frac{\beta \mu_1 (S_C) +1}{\beta_\ell \mu_1 (S_C) +1},
	\end{equation} 
	$S_C = F^T C^{-1} F$, and $\mu_1(S_C)$ is the spectral radius of $S_C$.
\end{theorem}

\begin{proof}
	The generalized eigenvalue problem \eqref{eq:geneig} can be rewritten as:
	\begin{equation}\label{eq:geneig1}
		\left [ C (1 - \lambda_i) + (\beta - \beta_\ell \lambda_i) F F^T \right] \mybold{v}_i = \mybold{0}.
	\end{equation}
	Setting $H = C^{-1} F F^T$ and  $\mu_i = (\lambda_i - 1) / (\beta -  \beta_\ell \lambda_i)$, equation 
	\eqref{eq:geneig1} reads:
	\begin{equation} \label{eq:eigH}
		H \mybold{v}_i = \mu_i  \mybold{v}_i.
	\end{equation}
	The eigenvalues $\mu_i$ of $H$ are either 0, with multiplicity $n - n_p$, or equal to those of $S_C$. Furthermore, the eigenvalues 
	$\lambda_i$ are related to $\mu_i$ as follows:
	\begin{equation}\label{eq:lam_i}
		\lambda_i= \frac{\beta \mu_i +1}{\beta_\ell \mu_i +1}.
	\end{equation}
	Since $\mu_i\in[0,\mu_1(S_C)]$ and $\lambda_i$ in \eqref{eq:lam_i} increases monotonically with $\mu_i$ for
	$\beta>\beta_\ell$, the eigenvalues $\lambda_i $ belong to the interval $\mathcal{I}$: 
	\begin{equation}\label{eq:mathI} 
		\mathcal{I} = \left \{  \lambda_i \in \mathbb{R} \; | \; 1 \le \lambda_i \le \frac{\beta \mu_1 (S_C) +1}
		{\beta_\ell \mu_1 (S_C) +1} \right \}.
	\end{equation}
\end{proof}

Let us denote with $\hat{C}_\ell$ the matrix
\begin{equation}
  \hat{C}_\ell = C+\beta_\ell FF^T,
  \label{eq:C_ell}
\end{equation}
obtained with the lower-bound value $\beta_\ell$ for $\beta$,
and use $\hat{C}_\ell^{-1}$ as a preconditioner for the inner solve with $\hat{C}$ in the RPF application (Algorithm 
\ref{alg:appRPF}). Theorem \ref{thm:thm1} shows that the spectral conditioning number of $\hat{C}_\ell^{-1}\hat{C}$
grows indefinitely with $\beta$, making the inner solve with $\hat{C}$ more difficult and expensive as
$\beta/\beta_\ell\rightarrow\infty$. If the solution to the inner system \eqref{eq:hat_Cu_b} is obtained inexactly,
as usually necessary for the sake of computational efficiency in large-size applications, the accuracy of $\mybold{w}$
is expected to decrease when the separation between $\beta$ and the limiting value $\beta_\ell$ grows.
This may lead to an
overall degradation of the RPF performance that might even prevent the outer Krylov solver from converging.

{
In order to stabilize the global RPF behavior, it is therefore important to restore equations \eqref{eq:Khat} and \eqref{eq:Ahat} in the preconditioner construction and enhance the accuracy in the local solve with $\hat{C}$. To this aim, we introduce here two different approaches:
}
\begin{itemize}
	\item Method 1: exploits the structure of $\hat{C}$ by defining a natural splitting that can be used to 
		develop an effective preconditioner for the inner problem;
	\item Method 2: uses an appropriate projection operator 
	in order to get rid of the
		near-null space of $\hat{C}$.
\end{itemize} 
The use of Method 1 and Method 2 gives rise to the Enhanced RPF preconditioners (ERPF1 and ERPF2, respectively).
{We notice here that both Method 1 and Method 2 are fully algebraic and are developed for a general SPD matrix $C$ and rank-deficient term $FF^T$. Hence, their generalization to applications beyond the one focussed in this work is reasonably straightforward.}

\subsection{Method 1}\label{subsec:met1}
A weighted splitting of matrix $\hat{C}$ is adopted to obtain a stationary iterative scheme. Let us write $\hat{C}$
as follows:
\begin{equation}\label{eq:hatC_split}
\hat{C} = \left(1-\frac{\beta}{\beta_\ell}\right) C + \frac{\beta}{\beta_\ell} \left(C+ \beta_\ell F F^T\right) = \left(1-\frac{\beta}{\beta_\ell}\right) C + \frac{\beta}{\beta_\ell} \hat{C}_\ell, 
\end{equation}
and consider the following stationary iteration to solve \eqref{eq:hat_Cu_b}:
\begin{equation}\label{eq:itsche}
{\mybold{w}_{k+1} = \mybold{w}_{k} + \frac{\beta_\ell}{\beta} \hat{C}_\ell^{-1} \mybold{r}_k, }
\end{equation}
with $\mybold{r}_k = (\mybold{b}  - \hat{C} \mybold{w}_{k} )$ the residual at iteration $k$.
The iteration matrix $G$ associated with scheme \eqref{eq:itsche} reads:
\begin{equation}\label{eq:itMat}
{G = I - \frac{\beta_\ell}{\beta} \hat{C}_\ell^{-1}\hat{C}  
= \left( 1 - \frac{\beta_\ell}{\beta} \right) \hat{C}_\ell^{-1}  C.} 
\end{equation}

\begin{lemma}\label{lem:m1}
The stationary scheme \eqref{eq:itsche} for the solution of the system
$\hat{C}\mybold{w}=\vec{b}$ of equation \eqref{eq:hat_Cu_b}, with  $\beta >\beta_\ell>0$, is unconditionally convergent with rate $R = \log(1-\beta_\ell/\beta)$. 
\end{lemma}

\begin{proof}
	The spectral radius $\lambda_1(G)$ of the iteration matrix is given by:
	\begin{equation}\label{eq:l1G}
		\lambda_1(G) = \left (1 - \frac{\beta_\ell}{\beta}\right ) \lambda_1({N}),
	\end{equation}
	with $N = \hat{C}_\ell^{-1} C$ and $\lambda_1(N)$ its spectral radius. Since $N$ reads
	\begin{equation}\label{eq:matE}
		N = \left( C+ \beta_\ell F   F^T \right)^{-1} C = \left[ C \left( I + \beta_\ell C^{-1} F F^T \right)
		\right]^{-1} C = \left( I + \beta_\ell H \right)^{-1},
	\end{equation}
	its spectral radius is:
	\begin{equation} \label{eq:l1E}
		\lambda_1(N) = \frac{1}{ 1+ \beta_\ell \lambda_n(H )},
	\end{equation}
	where $\lambda_n(H)$ denotes the smallest eigenvalue of $H=C^{-1}FF^T$. Being $H$ similar to the symmetric positive semidefinite 
	matrix $C^{-\frac{1}{2}} F F^T C^{-\frac{1}{2}}$, 
	$\lambda_n(H)=0$ and 
	\begin{equation} \label{eq:l1G_fin}
		\lambda_1(G) = 1 - \frac{\beta_\ell}{\beta} < 1.
	\end{equation}
	Hence, the stationary scheme \eqref{eq:itsche} is unconditionally convergent with rate $R = \log(1-\beta_\ell/\beta)$.
\end{proof}

The idea of Method 1 is to solve the system of equation \eqref{eq:hat_Cu_b} with the scheme \eqref{eq:itsche}. Since
in the RPF preconditioner application such system is solved inexactly, we carry out a pre-defined number of stationary
iterations, $n_{in}$, and replace the exact application of $\hat{C}_\ell^{-1}$ with an inexact one, say the inner preconditioner
$M_{\hat{C}_\ell}^{-1}$. This procedure is summarized in Algorithm \ref{alg:met1}.

\begin{algorithm}[tbh]
	\caption{Method 1: $\mybold{w}=\mbox{\sc MET1\_Apply}\left(n_{in},\beta,\beta_\ell,\vec{b},\hat{C},M^{-1}_{\hat{C}_\ell}\right)$}
	\label{alg:met1}
    \begin{algorithmic}[1]
		\State $\mybold{w} = \mybold{0}$
		\For {$k=1,\dots,n_{in}$}
		\State $\mybold{r} = \mybold{b} - \hat{C} \mybold{w}$
		\State \textbf{Apply} $M^{-1}_{\hat{C}_\ell}$ to $\mybold{r}$ to get $\mybold{v}$ 
		\State $\mybold{w} \gets \mybold{w} + (\beta_\ell/\beta) \mybold{v}$
		\EndFor
	\end{algorithmic}
\end{algorithm}

The ERPF1 preconditioner is obtained by replacing lines 3 and 8 of Algorithm \ref{alg:appRPF} with the function
in Algorithm \ref{alg:met1} whenever $\alpha<\alpha_K$ and $\alpha<\alpha_A$, respectively.
{
For each inner iteration 
$k$, Method 1 involves one matrix-vector product, one application of $M^{-1}_{\hat{C}_\ell}$ and two vector updates, hence the algorithm
cost can grow up quickly with $n_{in}$. More precisely, denoting by $n$ the number of rows of $\hat{C}$, the complexity $\chi_1$ of Algorithm \ref{alg:met1} reads:
\begin{equation}
    \chi_1 = n_{in} \times \left[ \chi \left( M^{-1}_{\hat{C}_\ell} \right) + \mathcal{O} \left(n\right) \right],
    \label{eq:comp_met1}
\end{equation}
where $\chi(M^{-1}_{\hat{C}_\ell})$ is the complexity of the inner preconditioner application.
}
For this reason, $n_{in}$ should be kept as low as possible, say 2 or 3.

{
\begin{remark}
\label{rem:speed_beta}
Equation \eqref{eq:l1G_fin} shows that the convergence rate of the iteration \eqref{eq:itsche} degrades to 0 as $\beta\rightarrow\infty$. In our application, $\beta$ is either $1/\alpha$ (equation \eqref{eq:Khat}) or $\gamma/\alpha$ (equation \eqref{eq:Ahat}). In a full poromechanical simulation, the condition $\gamma/\alpha\rightarrow\infty$ is more likely to occur and usually happens at the final stage of the process, when the solution approaches steady-state conditions and $\Delta t$, hence $\gamma$, progressively increases. In this situation, Method 1 is expected to lose efficiency because it may require too many inner iterations to obtain a sufficiently accurate result.
\end{remark}
}

\begin{remark}
    \label{rem:innerCG}
    {
    The stationary scheme \eqref{eq:itsche} is a Richardson iteration preconditioned by $(\beta_{\ell}/\beta)\hat{C}_{\ell}^{-1}$. Since $\hat{C}$ is SPD, it could be also effectively replaced by a Conjugate Gradient scheme with the same preconditioner. Stopping the inner iterations after $n_{in}=$ 2 or 3 steps, however, usually does not provide a significantly different outcome with respect to the use of Algorithm \ref{alg:met1}. 
    }
\end{remark}

\subsection{Method 2}\label{subsec:met2}
{
Another strategy to solve \eqref{eq:hat_Cu_b} relies on using an appropriate projection operator. In particular, the idea
is to project the system \eqref{eq:hat_Cu_b} onto a space $\mathcal{Z}$ such that:
\begin{equation}
    \mathbb{R}^n = \ker(F^T) \oplus \mathcal{Z},
    \label{eq:Zdef}
\end{equation}
so as to annihilate the components
of $\mybold{w}$ lying in the kernel of $F^T$.
Let us define the $\hat{C}$-orthogonal projector $P$ as:
\begin{equation}\label{eq:proj}
   P = Z \left[ Z^T \left( \hat{C} Z \right) \right]^{-1} Z^T \hat{C},
\end{equation}
where $Z=[\vec{z}_1,\dots,\vec{z}_{n_p}] \in \mathbb{R}^{n \times n_p}$ is the generator of $\mathcal{Z}$:
\begin{equation}\label{eq:spaceZ}
	\mathcal{Z} =\operatorname{span}\{\vec{z}_1,\dots,\vec{z}_{n_p}\}. 
\end{equation}
The operator $P$ is the linear mapping that projects a vector $\mybold{w}$ onto $\mathcal{Z}$ orthogonally to $\mathcal{L}$:
\begin{equation}\label{eq:spaceS}
   \mathcal{L} = \left \{ \mybold{y}_i \in \mathbb{R}^n \; | \; \mybold{y}_i = \hat{C} \mybold{z}_i, \; \mybold{z}_i \in  
	\mathcal{Z}    \right \}.
\end{equation}
Notice that the projector $P$ is $\hat{C}$-orthogonal according to the following definition \cite{Hac16}:
\begin{defin}
A projector $P$ onto a subspace $\mathcal{Z}$ is $\hat{C}$-orthogonal if and only if it is self-adjoint with respect 
to the $\hat{C}$-inner product.
\end{defin}
Of
course, we have
  $\ran{P} = \ran{Z}$.
The solution $\mybold{w}$ of \eqref{eq:hat_Cu_b} can be decomposed as the sum of two contributions:
\begin{equation}\label{eq:splitupu}
  \mybold{w} = P \mybold{w} + \left ( I - P \right ) \mybold{w} = \mybold{w}_{\mathcal{Z}} + \mybold{w}_{\mathcal{S}},
\end{equation}
where $\mybold{w}_{\mathcal{Z}}\in\mathcal{Z}$ and $\mybold{w}_{\mathcal{S}}\in\mathcal{S}=\ran{I-P}$. 
Applying the
projector \eqref{eq:proj}, the component $\mybold{w}_{\mathcal{Z}}$ is 
given by:
\begin{equation}\label{eq:u_Z}
   \mybold{w}_{\mathcal{Z}}  = P \mybold{w} = Z E^{-1} Z^T \hat{C} \mybold{w} =Z E^{-1} Z^T \mybold{b},
\end{equation}
where $E=Z^T\hat{C}Z$ is the projection of $\hat{C}$ onto $\mathcal{Z}$. In order to obtain $\mybold{w}_{\mathcal{S}}$, 
let us consider the following projected system:
\begin{equation}\label{eq:proSystem}
   \left ( I - P^T \right ) \hat{C}  \mybold{w} = \left ( I - P^T \right ) \mybold{b}.
\end{equation}
Recalling that $(I-P^T) \hat{C}=\hat{C} (I - P)$, 
equation \eqref{eq:proSystem} can be rewritten as:
\begin{equation}\label{eq:u_M1}
  \hat{C} \mybold{w}_{\mathcal{S}} = \left ( I - P^T \right ) \mybold{b}.
\end{equation}
Since $\mybold{w}_{\mathcal{S}}\perp\mathcal{L}$ of equation \eqref{eq:spaceS}, the following condition holds:
\begin{equation}
    Z^T \left( C + \beta F F^T \right) \mybold{w}_{\mathcal{S}} = 0.
    \label{eq:ortho_wS}
\end{equation}
If we set $Z=C^{-1}F$, equation \eqref{eq:ortho_wS} becomes:
\begin{equation}
    \left( I + \beta 
    S_C \right) F^T \mybold{w}_{\mathcal{S}} = 0,
    \label{eq:wS_kerFt}
\end{equation}
which implies $\mybold{w}_{\mathcal{S}}\in\ker(F^T)$ because of the regularity of the matrix $(I + \beta S_C)$, with $S_C=F^TC^{-1}F$ (see proof of Theorem \ref{thm:thm1}).
Hence, by setting $\mathcal{Z}=\ran{C^{-1}F}$ equation \eqref{eq:u_M1} simplifies to:
\begin{equation}\label{eq:u_M}
  C \mybold{w}_{\mathcal{S}} = \left ( I - P^T \right ) \mybold{b},
\end{equation}
which no longer requires the inversion of $\hat{C}$.
Observe that the contribution $P^T \vec{b}$ at the right-hand side of \eqref{eq:u_M} reads:
\begin{equation} \label{eq:bZ}
	\hat{C} Z E^{-1} Z^T \vec{b} = \hat{C} \mybold{w}_{\mathcal{Z}},
\end{equation}
i.e., $P^T\vec{b}=\hat{C}\mybold{w}_{\mathcal{Z}}=\vec{b}_{\mathcal{Z}}$ is the projected right-hand side onto $\mathcal{Z}$. 
The solution $\mybold{w}$ to the system \eqref{eq:hat_Cu_b} is therefore given by the following set of equations:
\begin{eqnarray}
 \mybold{w}_{\mathcal{Z}}  &=& Z {E}^{-1} Z^T \mybold{b}, \nonumber \\
  \mybold{w}_{\mathcal{S}}  &=&  {C}^{-1} \left (  \mybold{b} - \mybold{b}_{\mathcal{Z}} \right), \label{eq:syspromet2}  \\
  \mybold{w}_{\;\;} &=& \mybold{w}_{\mathcal{Z}} + \mybold{w}_{\mathcal{S}}. \nonumber 
\end{eqnarray} 
}
The matrix $ZE^{-1}Z^T$ reads:
\begin{eqnarray}
	Z E^{-1} Z^T &=& C^{-1} F \left[ F^T C^{-1} C  C^{-1} F + \beta \left( F^T C^{-1} F F^T C^{-1} F  \right ) 
	\right]^{-1} F^T C^{-1} \nonumber \\
		    &=& C^{-1} F \left[ S_C  + \beta S_C^2 \right]^{-1} F^T C^{-1} \nonumber \\
		    &=& C^{-1} F S_C^{-1} \left[ I + \beta S_C \right]^{-1} F^T C^{-1}. \label{eq:ZEZ}
\end{eqnarray}
Introducing \eqref{eq:ZEZ} in the first equation of \eqref{eq:syspromet2}, we have:
\begin{equation}\label{eq:uz_CF}
 \mybold{w}_{\mathcal{Z}}  = C^{-1} F S_C^{-1} \tilde{S}_C^{-1} F^T C^{-1} \mybold{b},
\end{equation} 
with $\tilde{S}_C= I + \beta S_C.$ The global solution $\mybold{w}$ can be rewritten as:
\begin{eqnarray}\label{eq:u_uz}
\mybold{w} &=& \mybold{w}_{\mathcal{S}} + \mybold{w}_{\mathcal{Z}}= {C}^{-1} \left (  \mybold{b} - \mybold{b}_{\mathcal{Z}} \right) + \mybold{w}_{\mathcal{Z}} \nonumber \\
	&=&  {C}^{-1} \left (  \mybold{b} - \hat{C}  \mybold{w}_{\mathcal{Z}} + C \mybold{w}_{\mathcal{Z}} \right) = {C}^{-1} \left (  \mybold{b} -  \beta F F^T  \mybold{w}_{\mathcal{Z}} \right),
\end{eqnarray}
and using $\mybold{w}_{\mathcal{Z}}$ of equation \eqref{eq:uz_CF} finally yields:
 \begin{equation}\label{eq:precSMW}
  \mybold{w} = {C}^{-1} \left(  \mybold{b} - \beta F F^T \mybold{w}_{\mathcal{Z}} \right)= {C}^{-1} \left (  \mybold{b} - \beta F \tilde{S}_C^{-1} F^T C^{-1} \mybold{b}  \right) .
\end{equation} 

\begin{remark}\label{rem:SMW}
	The result \eqref{eq:precSMW} can be also obtained applying directly the Sherman-Morrison-Woodbury identity 
	to equation \eqref{eq:hat_Cu_b}. Therefore, this relationship can be also regarded as the natural outcome
	of a particular projection operated on the matrix $\hat{C}$.
\end{remark}

The solution of the inner system \eqref{eq:hat_Cu_b} through equation \eqref{eq:precSMW} requires two inner solves with
$C$, which is a well-defined SPD matrix, and one with $\tilde{S}_C$. The latter is not available explicitly, however
it can be easily and inexpensively approximated for $\hat{K}$ and $\hat{A}$ by using, for instance, $D_K$ and $D_A$ 
already adopted to compute $\alpha$ (see Section \ref{sec:RPF}). Hence, an inexact solve with $\tilde{S}_C$ is performed
by using some approximation, say $M_{\tilde{S}}^{-1}$. Of course, also the solves with $C$ can be performed inexactly,
by means of another inner local preconditioner, $M_C^{-1}$, for SPD matrices. Algorithm \ref{alg:met2} summarizes the
procedure resulting from Method 2.

\begin{algorithm}[tbh]
	\caption{Method 2: $\mybold{w}=\mbox{\sc MET2\_Apply}\left(\beta,\vec{b},F,M^{-1}_C,M^{-1}_{\tilde{S}}\right)$}
        \label{alg:met2}
        \begin{algorithmic}[1]
                \State \textbf{Apply} $M^{-1}_C$ to $\mybold{b}$ to get $\mybold{c}$
                \State $\mybold{d}=F^T \mybold{c}$
		\State \textbf{Apply} $M^{-1}_{\tilde{S}}$ to $\mybold{d}$ to get $\mybold{g}$
                \State $\mybold{c} = F \mybold{g}$
		\State $\mybold{c} \gets \mybold{b}-\beta \mybold{c}$
                \State \textbf{Apply} $M^{-1}_C$ to $\mybold{c}$ to get $\mybold{w}$
        \end{algorithmic}
\end{algorithm}

The ERPF2 preconditioner is obtained by replacing lines 3 and 8 of Algorithm \ref{alg:appRPF} with the function
in Algorithm \ref{alg:met2} whenever $\alpha<\alpha_K$ and $\alpha<\alpha_A$, respectively.
{The complexity $\chi_2$ of Algorithm \ref{alg:appRPF} reads:
\begin{equation}
    \chi_2 = 2\chi \left( M^{-1}_C \right) + \chi \left( M^{-1}_{\tilde{S}} \right) + \mathcal{O} \left( n \right),
    \label{eq:comp_met2}
\end{equation}
where $n$ is the number of rows of $F$, and $\chi(M^{-1}_C)$ and $\chi(M^{-1}_{\tilde{S}})$ denote the complexity of the inner preconditioner applications.
}
The overall cost
for the preconditioner application per iteration increases with respect to the native RPF, however the overall
algorithm is expected to take benefit from the acceleration of the global Krylov solver.

\begin{remark}\label{rem:BTP}
	The use of Method 2 for solving the inner systems with both $\hat{K}$ and $\hat{A}$ is equivalent to compute
	the vector $\vec{t}=\blkMat{M}^{-1}\vec{r}$ using directly equation \eqref{eq:RPF}:
	\begin{equation}\label{eq:dirRPF}
		\left\{ \begin{array}{l} 
			\vec{y}=\alpha \blkMat{M}_1^{-1} \vec{r} \\
			\vec{t}=\blkMat{M}_2^{-1} \vec{y}
		\end{array} \right. ,
	\end{equation}
	where $\blkMat{M}_1$ and $\blkMat{M}_2$ are inverted with the aid of the factorizations:
	\begin{eqnarray}
		\blkMat{M}_1 & = & \left[ \begin{array}{ccc}
			K & 0 & -Q \\ 0 & \alpha I_q & 0 \\ Q^T & 0 & \alpha I_p
		\end{array} \right] = \left[ \begin{array}{ccc}
			K & 0 & 0 \\ 0 & \alpha I_q & 0 \\ Q^T & 0 & \tilde{S}_K
		\end{array} \right] \left[ \begin{array}{ccc}
			I_u & 0 & -K^{-1} Q \\ 0 & I_q & 0 \\ 0 & 0 & \alpha I_p
		\end{array} \right], \label{eq:fact_M1} \\
		\blkMat{M}_2 & = & \left[ \begin{array}{ccc}
                        I_u & 0 & 0 \\ 0 & A & -B \\ 0 & \gamma B^T & \alpha I_p
                \end{array} \right] = \left[ \begin{array}{ccc}
                        I_u & 0 & 0 \\ 0 & A & 0 \\ 0 & \gamma B^T & \tilde{S}_A
                \end{array} \right] \left[ \begin{array}{ccc}
			I_u & 0 & 0 \\ 0 & I_q & -A^{-1} B \\ 0 & 0 & \alpha I_p
                \end{array} \right]. \label{eq:fact_M2}
	\end{eqnarray}
	This leads to an alternative ERPF2 implementation, which is naturally more prone to a split preconditioning 
	strategy.
	{The use of the factorizations \eqref{eq:fact_M1}-\eqref{eq:fact_M2} into \eqref{eq:dirRPF} whenever $\alpha<\alpha_K$ and $\alpha<\alpha_A$ provides the same numerical outcome as the application of Algorithm \ref{alg:met2}.}
	Details of such equivalence along with the alternative ERPF2 algorithm are provided in 
	\ref{app:altERPF2}.
\end{remark}

\section{Numerical Results}\label{sec:NumRes}
Two sets of numerical experiments are discussed in this section.
The first set (Test 1) arises from Mandel's problem \cite{Man53}, i.e., a classic benchmark of linear poroelasticity.
This problem is used to verify the theoretical properties of the proposed methods.
In the second set (Test 2), two real-world applications are considered in order to test the robustness and efficiency of the preconditioners. 

We consider three variants of RPF, ERPF1, and ERPF2 according to the selection of the inner preconditioners.
In essence, the first variant ($\blkMat{M}_{\textsc{i}}$) relies on applying ``exactly'' the inner preconditioners using nested direct solvers and aims at investigating the theoretical properties of the different approaches.
The second ($\blkMat{M}_{\textsc{ii}}$) and the third ($\blkMat{M}_{\textsc{iii}}$) variants introduce further levels of approximation utilizing incomplete Cholesky (\textsc{ic}) and algebraic multigrid (\textsc{amg}) preconditioners, respectively.
Specifically: \medskip

\begin{itemize}
	\item \textbf{RPF} (Algorithm \ref{alg:appRPF}): all exact/inexact solves involve $\hat{K}$ in \eqref{eq:Khat} and $\hat{A}$ in \eqref{eq:Ahat} irrespective of the value of $\alpha$;
	\medskip
	\item \textbf{ERPF1} (Algorithm \ref{alg:met1}): Method 1 requires inner solves with matrices $\hat{K}_{\alpha}$ or $\hat{A}_{\alpha}$ that depend on the relaxation parameter $\alpha$ as follows
	\begin{equation}
	  \hat{K}_{\alpha} = 
	  \begin{dcases}
	    \hat{K}_{\ell}, &\text{if } \alpha < \alpha_K   \\
	    \hat{K},        &\text{if } \alpha \ge \alpha_K \\
	  \end{dcases},
	  \qquad	  
	  \hat{A}_{\alpha} = 
	  \begin{dcases}
	    \hat{A}_{\ell}, &\text{if } \alpha < \alpha_A   \\
	    \hat{A},        &\text{if } \alpha \ge \alpha_A \\
	  \end{dcases},
	\end{equation}
	with $\hat{K}_{\ell}$ and $\hat{A}_{\ell}$ computed at $\alpha_K$ and $\alpha_A$, respectively, using expression \eqref{eq:C_ell};
	\medskip
	\item \textbf{ERPF2} (Algorithm \ref{alg:met2}): Method 2 is used to replace the inner solves with $\hat{K}$ or $\hat{A}$ whenever either $\alpha<\alpha_K$ or $\alpha<\alpha_A$. It requires inner solves with $K$ and $\tilde{S}_K$, or $A$ and $\tilde{S}_A$. For $\tilde{S}_K$ we use the diagonal matrix $D_K$ employed in the computation of $\alpha$ and $\alpha_K$	 (see Section \ref{sec:RPF}) by defining:
				\begin{equation} \label{eq:tildeDK}
					\tilde{D}_K = I_p + \frac{1}{\alpha} D_K.
				\end{equation}
				The application of $\tilde{S}_K^{-1}$ is approximated by $\tilde{D}_K^{-1}$. For $\tilde{S}_A$
				we use the diagonal matrix $\tilde{A}^{-1}$ employed in the computation of $\alpha$ and 
				$\alpha_A$ (see Section \ref{sec:RPF}, equation \eqref{eq:tilde_A}) by defining:
                                \begin{equation} \label{eq:tildeSA}
					\tilde{S}_A \simeq I_p + \frac{\gamma}{\alpha} B^T \tilde{A}^{-1} B.
                                \end{equation}
				The inverse of $\tilde{A}$ is used to approximate the application of $A^{-1}$ as well.
\end{itemize}
Table \ref{tab:varPrec} summarizes the different variants illustrated above. For the incomplete Cholesky factorization, 
we use an algorithm implementing a fill-in technique based on the selection of a user-specified row-wise number of 
entries in addition to those of the original matrix, as proposed in \cite{Saa94} and \cite{LinMor99}. A classic algebraic 
multigrid method \cite{RugStu87} is used as implemented in the \texttt{HSL\_MI20} 
package \cite{BoyMihSco10}. Of course, several other algebraic options for the inner blocks are possible.

\begin{table}
\small
\caption{RPF, ERPF1, and ERPF2 preconditioner variants.\label{tab:varPrec}}
\begin{center}
\begingroup
\begin{tabular}{ccccc}
\toprule
Preconditioner &
RPF &
ERPF1 &
\begin{tabular}{@{}c@{}}ERPF2 \\ $\alpha < \alpha_K$ \end{tabular} &
\begin{tabular}{@{}c@{}}ERPF2 \\ $\alpha < \alpha_A$ \end{tabular}  \\
\midrule
$\blkMat{M}^{-1}_\textsc{i}$ &
$\blkMat{M}^{-1}_\textsc{i}(\hat{K}^{-1},\hat{A}^{-1})$ &
$\blkMat{M}^{-1}_\textsc{i}(\hat{K}^{-1}_{\alpha},\hat{A}^{-1}_{\alpha})$ &
$\blkMat{M}^{-1}_\textsc{i}({K}^{-1},\tilde{D}^{-1}_{K},\hat{A}^{-1})$ &
$\blkMat{M}^{-1}_\textsc{i}(\hat{K}^{-1},\tilde{A}^{-1},\tilde{S}_A^{-1})$ \\
$\blkMat{M}^{-1}_\textsc{ii}$ &
$\blkMat{M}^{-1}_\textsc{ii}(\hat{K}^{-1}_\textsc{ic},\hat{A}^{-1}_\textsc{ic})$ &
$\blkMat{M}^{-1}_\textsc{ii}(\hat{K}_{\alpha,\textsc{ic}}^{-1},\hat{A}^{-1}_{\alpha,\textsc{ic}})$ &
$\blkMat{M}^{-1}_\textsc{ii}({K}_\textsc{ic}^{-1},\tilde{D}^{-1}_{K},\hat{A}^{-1}_\textsc{ic})$ &
$\blkMat{M}^{-1}_\textsc{ii}(\hat{K}^{-1}_\textsc{ic},\tilde{A}^{-1},\tilde{S}_{A, \textsc{ic}}^{-1})$ \\
$\blkMat{M}^{-1}_\textsc{iii}$ &
$\blkMat{M}^{-1}_\textsc{iii}(\hat{K}^{-1}_\textsc{amg},\hat{A}^{-1}_\textsc{amg})$ &
 -- &
$\blkMat{M}^{-1}_\textsc{iii}({K}_\textsc{amg}^{-1},\tilde{D}^{-1}_{K},\hat{A}^{-1}_\textsc{amg})$ &
$\blkMat{M}^{-1}_\textsc{iii}(\hat{K}^{-1}_\textsc{amg},\tilde{A}^{-1},\tilde{S}_{A, \textsc{amg}}^{-1})$ \\
\bottomrule
\end{tabular}
\endgroup
\end{center}
\end{table}

{The objective of this section is to demostrate numerically that the proposed enhancements ERPF1 and ERPF2 overcome the drawbacks of the native RPF preconditioner. A comparison with other well-established block preconditioners, which has already been done for RPF in~\cite{FriCasFer19}, is not performed here. 
Test 1 shows the stability and robustness of ERPF1 and ERPF2 to a variation of the mesh and timestep size for tightly coupled poromechanical conditions. Tests 2 give the performance in real-world heterogeneous problems, where the robustness to significant jumps in the hydromechanical parameters and the computational efficiency measured in terms of CPU time reduction are the main issues.  
}

In all test cases, Bi-CGStab \cite{Vor92} is selected as a Krylov subspace method. 
{Our focus is the robustness and computational efficiency in the solution of a single linear system. 
Hence, following standard practice \cite{Hac16}, all tests are performed by setting a random solution vector $\Vec{x}$ and computing the corresponding right-hand side $\Vec{F} = \blkMat{A} \Vec{x}$.}
The iterations are stopped when the 
ratio between the 2-norm of the real residual vector and the 2-norm of the right-hand side
is smaller than $\tau=10^{-6}$. 
The computational performance is evaluated in terms of number of iterations $n_{it}$, 
CPU time in seconds for the preconditioner set-up $T_p$ and for the solver to converge $T_s$. The total time is 
denoted by $T_t = T_p + T_s$. All computations are performed using a code written in FORTRAN90 on an Intel(R) Xeon(R) 
CPU E5-1620 v4 at 3.5 GHz Quad-Core with 64 GB of RAM memory.

\subsection{Test1: Mandel's problem}\label{subsec:Man1}
This is a classical benchmark for validating coupled poromechanical models. {The problem consists of a porous slab with section $2a\times2b$,
discretized by a cartesian grid, sandwiched between rigid, frictionless, and impermeable plates (Figure \ref{fig:Man}). For the sake of simplicity, we set $a=b$}. 
A detailed description of the test case is reported in \cite{CasWhiFer16,FriCasFer19}. This test case has 
mainly a theoretical value and is used to investigate the properties of ERPF1 and ERPF2 for a wide range of time step, 
$\Delta t$, and characteristic mesh size, $h$, values. In particular, four grids progressively refined are considered 
as shown in Table \ref{tab:Mangrids}. 
{
The time-step size is defined with respect to the characteristic consolidation time $t_c$ of the problem. This allows for testing the solver behavior at any stage of the consolidation process independently of the specific material parameters, i.e., the medium compressibility and permeability. Hence, the analysis carried out in this test case provides also information about the parameter-robustness of the proposed algorithms.
}


\begin{figure}
  \hfill
  \begin{subfigure}[t]{0.22\linewidth}
    \centering
    \includegraphics[width=\linewidth]{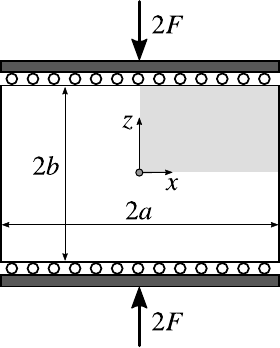}
    \caption{}
  \end{subfigure}
  \hfill
  \begin{subfigure}[t]{0.3\linewidth}
    \centering
    \includegraphics[width=\linewidth]{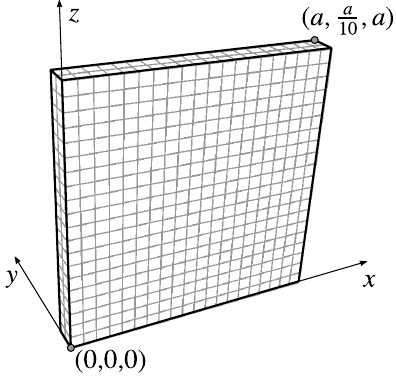}
    \caption{}
  \end{subfigure}
  \hfill\null
  \caption{Test 1, Mandel's problem: sketch of the geometry (a) and the computational model (b).}
  \label{fig:Man}
\end{figure}


\begin{table}
  \small
  \caption{Test 1, Mandel's problem: grid refinement and problem size. 
  }\label{tab:Mangrids}
  \begin{center}
    \begin{tabular}{ c c r r r }
      \toprule
      \multicolumn{1}{c}{$a/ h$}& \begin{tabular}{@{}c@{}} number of  \\ elements \end{tabular} &  \multicolumn{1}{r}{$n_u$}&  \multicolumn{1}{r}{$n_q$}& 
      \multicolumn{1}{r}{$n_p$}\\ 
      \midrule
      10 &  $10 \times 1 \times 10$ &    726 &    420 &   100\\ 
      20 &  $20 \times 2 \times 20$ &   3,969 &   2,880 &   800\\ 
      40 &  $40 \times 4 \times 40$ &  25,215 &  21,120 &  6,400\\ 
      80 &  $80 \times 8 \times 80$ & 177,147 & 161,280 & 51,200\\ 
      \bottomrule
    \end{tabular} 
  \end{center}  
\end{table}

\begin{figure}
  \hfill
  \begin{subfigure}[b]{0.48\textwidth}
    \centering
    \includegraphics[scale=1]{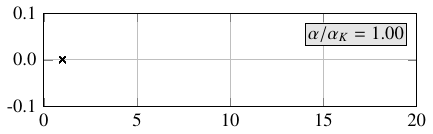}
    \includegraphics[scale=1]{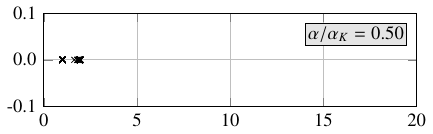}
    \includegraphics[scale=1]{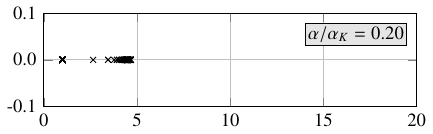}
    \includegraphics[scale=1]{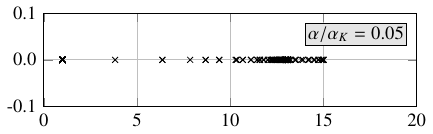}
    \caption{}
    \label{fig:KeigThe}
  \end{subfigure}
  \hfill
  \begin{subfigure}[b]{0.48\textwidth}
    \centering
    \includegraphics[scale=1]{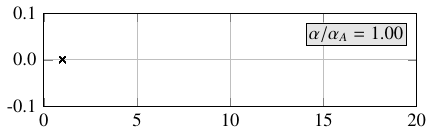}
    \includegraphics[scale=1]{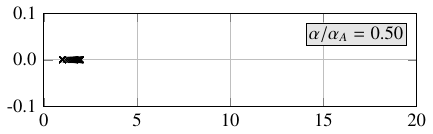}
    \includegraphics[scale=1]{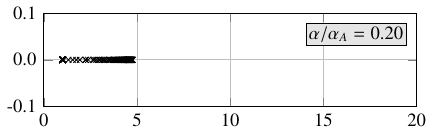}
    \includegraphics[scale=1]{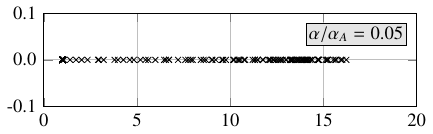}
     \caption{}
     \label{fig:AeigThe}
   \end{subfigure}
   \hfill\null
  \caption{Test 1, Mandel's problem ($a/h=10$): eigenspectrum of $\hat{K}_\ell^{-1} \hat{K}$ and $\hat{A}_\ell^{-1} \hat{A}$ 
  varying $\alpha/\alpha_K$ (a) and $\alpha/\alpha_A$ (b).}
  \label{fig:eig_spe}
\end{figure}

First of all, we analyze how the eigenspectrum of $\hat{K}_\ell^{-1} \hat{K}$ and $\hat{A}_\ell^{-1} \hat{A}$ changes 
with the ratios $\alpha/\alpha_K$ and $\alpha/\alpha_A$, respectively. Figure \ref{fig:eig_spe} provides such 
eigenspectra for the coarsest discretization ($a/h=10$). 
As the separation between $\alpha$ and the limiting values
$\alpha_K$ and $\alpha_A$ increases, the largest eigenvalue progressively grows and the eigenspectrum is less and less
clustered, as theoretically predicted by the result of Theorem \ref{thm:thm1}. 
{This is basically due to the increase of the condition number of $\hat{K}$ and $\hat{A}$ with respect to that of $K$ and $A$, respectively. The numerical values for the cases reported in Figure \ref{fig:eig_spe} are provided in Table \ref{tab:illcond}, while the general behavior obtained varying $\alpha$ is shown in Figure \ref{fig:illcond}. In other words, $\hat{K}$ and $\hat{A}$ are getting closer and closer to a singular matrix as $\alpha\rightarrow 0$, and consequently} 
the global RPF 
performance is expected to get worse. The application of Method 1 and Method 2 allows to overcome such an issue. 

\begin{table}
  \small
  \caption{{Test 1, Mandel's problem ($a/h=10$): condition number $\kappa$ of $\hat{K}$ and $\hat{A}$ varying $\alpha/\alpha_{\ell}$, $\alpha_{\ell}$ being either $\alpha_K$ or $\alpha_A$. The condition number of $K$ and $A$ are: $\kappa(K)=2.14\cdot10^5$, $\kappa(A)=9.84\cdot10^9$.}}\label{tab:illcond}
  \begin{center}
    \begin{tabular}{ r r r }
      \toprule
      {$\alpha/\alpha_{\ell}$} & {$\kappa(\hat{K})$} & {$\kappa(\hat{A})$} \\ 
      \midrule
       {1.00} &  {$8.47\cdot10^6$}  &  {$1.21\cdot10^{12}$} \\
       {0.50} &  {$1.67\cdot10^7$}  &  {$2.42\cdot10^{12}$} \\
       {0.20} &  {$4.15\cdot10^7$}  &  {$6.05\cdot10^{12}$} \\
       {0.05} &  {$1.65\cdot10^8$}  &  {$2.41\cdot10^{13}$} \\
      \bottomrule
    \end{tabular} 
  \end{center}  
\end{table}

\begin{figure}
    \centering
    \includegraphics[scale=1]{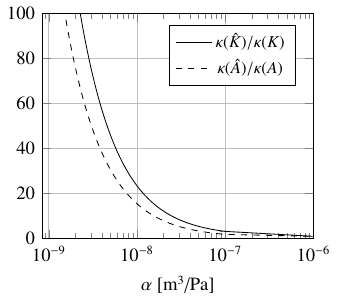}
    \caption{{Test 1, Mandel's problem ($a/h=10$): variation of the condition number $\kappa$ of $\hat{K}$ and $\hat{A}$ with respect to that of $K$ and $A$, respectively, vs $\alpha$.}}
    \label{fig:illcond}
\end{figure}

The preconditioner variant $\blkMat{M}^{-1}_{\textsc{i}}$ uses nested direct solvers, hence it has mainly a theoretical value.
It is useful to isolate the impact of the proposed schemes (Algorithms \ref{alg:met1} and \ref{alg:met2}) on the 
convergence behavior by varying the discretization steps in both space and time. We recall on passing that $\gamma=
\theta\Delta t$, with $\theta\in[0.5,1]$, and that $\alpha$ is proportional to $\sqrt{\gamma}$, $\alpha_K$ is constant
with $\gamma$, and $\alpha_A$ depends linearly on $\gamma$ (see Section \ref{sec:RPF}). Hence: 
\begin{equation}\label{eq:limDt}
	\lim_{\Delta t \rightarrow 0} \frac{\alpha}{\alpha_K} = 0, \qquad \lim_{\Delta t \rightarrow +\infty} 
	\frac{\alpha}{\alpha_A} = 0,
\end{equation}
i.e., the conditions $\alpha<\alpha_K$ and $\alpha<\alpha_A$ are encountered for small and large time-step sizes,
respectively, and are not likely to be satisfied simultaneously.

\begin{table}
\caption{Test 1, Mandel's problem: number of outer iteration with ERPF1 and $\blkMat{M}^{-1}_{\textsc{i}}$ by varying the 
	mesh-size $h$, the time-step $\Delta t$ and the inner iterations $n_{in}$. The time-step size is relative
	to the characteristic consolidation time $t_c=900$ s \cite{CasWhiFer16,FriCasFer19}.}
\small
\begin{center}
\begin{tabular}{lcccccccc}
\toprule
	\multicolumn{1}{c}{$a/ h$}&  \multicolumn{1}{c}{$\Delta t/ t_c$} & \multicolumn{1}{c}{$\alpha / \alpha_K$} & \multicolumn{1}{c}{$\alpha / \alpha_A$}&  \multicolumn{5}{c}{\# of outer iterations } \\
      \cmidrule{5-9}
	& & & & \multicolumn{1}{c}{$n_{in}={1}$}&\multicolumn{1}{c}{$n_{in}={2}$}&\multicolumn{1}{c}{$n_{in}={3}$}&\multicolumn{1}{c}{$n_{in}={4}$}&\multicolumn{1}{c}{$n_{in}={5}$}\\
      \midrule
	   &$5.6 \cdot 10^{-7}$ & 0.10 & $> 1$ & 12&9&7&5&3\\ 
	   &$3.4 \cdot 10^{-6}$ & 0.25 & $> 1$ & 10&7&4&3&3\\ 
	   &$1.7 \cdot 10^{-4}$ & 0.50 & $> 1$ & 7&4&4&3&3\\ 
	10 &$1.0 \cdot 10^{-3}$ & $> 1$  & $> 1$ & 8&--&--&--&--\\
	   &$5.0 \cdot 10^{-1}$ & $> 1$  & 0.50 &14&10&8&8&8\\ 
	   &$1.7 \cdot 10^{+0}$  & $> 1$  & 0.25 &22&12&10&9&8\\ 
	   &$5.6 \cdot 10^{+0}$  & $> 1$  & 0.10 &24&15&11&11&10\\ 
           &&&&&&\\    
	   &$5.6 \cdot 10^{-7}$ & 0.10 & $> 1$ & 12&9&6&4&3\\
	   &$3.4 \cdot 10^{-6}$ & 0.25 & $> 1$ & 9&7&4&3&2\\ 
	   &$1.7 \cdot 10^{-4}$ & 0.50 & $> 1$ & 7&4&3&3&3\\
	20 &$1.0 \cdot 10^{-3}$ & $> 1$  & $> 1$ & 7&--&--&--&--\\
	   &$5.0 \cdot 10^{-1}$ & $> 1$  & 0.50 &17&12&12&11&11\\
	   &$1.7 \cdot 10^{+0}$  & $> 1$  & 0.25 &29&16&13&13&12\\
	   &$5.6 \cdot 10^{+0}$  & $> 1$  & 0.10 &38&22&16&14&13\\
      \bottomrule
\end{tabular}
\end{center}
\label{tab:ManMet1}
\end{table}

Table \ref{tab:ManMet1} provides the outer iteration count obtained in Mandel's problem with ERPF1 by varying the mesh-size, 
the time-step and the number of inner iterations $n_{in}$. As expected, the number of outer iterations decreases when 
growing the number of inner iterations. Since the computational cost for ERPF1 application increases with $n_{in}$,
setting $n_{in}=2$ or $n_{in}=3$ appears to be already a good trade-off between solver acceleration and computational cost per
iteration.

The ERPF1 effectiveness is actually dependent on the value of $\alpha/\alpha_K$ or $\alpha/\alpha_A$. Lemma \ref{lem:m1}
shows that the convergence rate of the inner stationary method used with $\hat{K}$ or $\hat{A}$ tends to 0 with the
ratio $\alpha/\alpha_K$ or $\alpha/\alpha_A$. This is confirmed by Table \ref{tab:ManMet1}, which shows a performance
deterioration for both $\alpha/\alpha_K$ and $\alpha/\alpha_A$ approaching 0. Such a deterioration, however, appears to
be much more pronounced with $\alpha/\alpha_A$, i.e., for large values of the time-step $\Delta t$ {(see Remark \ref{rem:speed_beta})}. Table \ref{tab:ManMet1}
shows also a mild dependency on the mesh-size $h$, with the outer iteration count increasing moderately only for
large $\Delta t$ values.

\begin{table}
\caption{Test 1, Mandel's problem:  number of outer iterations with ERPF2 and $\blkMat{M}^{-1}_{\textsc{i}}$ by varying the mesh-size $h$ and the time-step $\Delta t$.\label{tab:ManMet2}}
\small
\begin{center}
\begin{tabular}{lccccccc}
\toprule
	\multicolumn{1}{c}{$a/ h$} &  \multicolumn{1}{c}{$\Delta t / t_c$}&  \multicolumn{1}{c}{$\alpha/\alpha_K$} &\multicolumn{1}{c}{\# of iter.}& &  \multicolumn{1}{c}{$\Delta t / t_c$}&  \multicolumn{1}{c}{$\alpha/\alpha_A$} &\multicolumn{1}{c}{\# of iter.}\\
      \midrule
	   &$10^{-6}$&$1.3 \cdot 10^{-1}$&5&&$10^{2}$&$3.5 \cdot 10^{-2}$&8 \\
	10 &$10^{-7}$&$4.3 \cdot 10^{-2}$&6&&$10^{3}$&$1.1 \cdot 10^{-2}$&8 \\
	   &$10^{-8}$&$1.4 \cdot 10^{-2}$&6&&$10^{4}$&$3.5 \cdot 10^{-3}$&7\\
	   &&&&&&&\\                                                      
	   &$10^{-6}$&$3.0 \cdot 10^{-1}$&6&&$10^{2}$&$1.9 \cdot 10^{-2}$&8 \\
	20 &$10^{-7}$&$9.5 \cdot 10^{-2}$&6&&$10^{3}$&$6.0 \cdot 10^{-3}$&7 \\
	   &$10^{-8}$&$3.0 \cdot 10^{-2}$&6&&$10^{4}$&$1.9 \cdot 10^{-3}$&7 \\
	   &&&&&&&\\                                                      
	   &$10^{-6}$&$1.9 \cdot 10^{-1}$&5&&$10^{2}$&$9.5 \cdot 10^{-3}$&9\\
	40 &$10^{-7}$&$6.0 \cdot 10^{-2}$&6&&$10^{3}$&$3.0 \cdot 10^{-3}$&7\\
	   &$10^{-8}$&$1.8 \cdot 10^{-2}$&6&&$10^{4}$&$9.5 \cdot 10^{-4}$& 7\\
	   &&&&&&&\\                       
	   &$10^{-6}$&$7.9 \cdot 10^{-1}$&6&&$10^{2}$&$6.2 \cdot 10^{-3}$&11\\
	80 &$10^{-7}$&$2.5 \cdot 10^{-1}$&6&&$10^{3}$&$1.9 \cdot 10^{-3}$&8\\
	   &$10^{-8}$&$8.0 \cdot 10^{-2}$&6&&$10^{4}$& $6.2 \cdot 10^{-4}$& 7\\
      \bottomrule
\end{tabular}
\end{center}
\end{table}

A similar analysis is carried out for ERPF2 with the $\blkMat{M}^{-1}_{\textsc{i}}$ variant. Table \ref{tab:ManMet2} provides the results 
in terms of outer iteration count by varying $\Delta t$ and $h$. This approach turns out to be stable with respect to
variations of both $\Delta t$ and $h$, also for extreme values of the time-step size. A comparison with Table
\ref{tab:ManMet1} reveals also that EPRF2 appears to be more effective than EPRF1, as far as the number of outer
iterations is concerned, especially when $\alpha<\alpha_A$, i.e., large time steps. 
{For this reason, we can conclude that ERPF2 should be usually preferred as a full poromechanical simulation proceeds towards steady-state conditions.}

\begin{figure}
  \hfill
  \begin{subfigure}[t]{0.475\linewidth}
    \centering
    \includegraphics[scale=1]{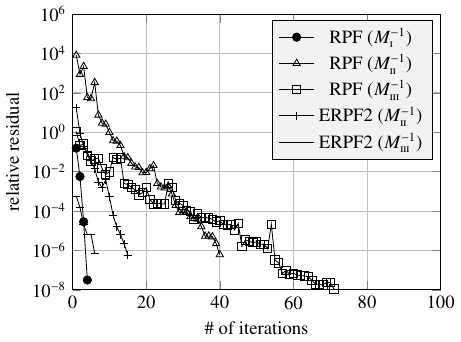} 
    \caption{}
  \end{subfigure}
  \hfill
  \begin{subfigure}[t]{0.475\linewidth}
    \centering    
    \includegraphics[scale=1]{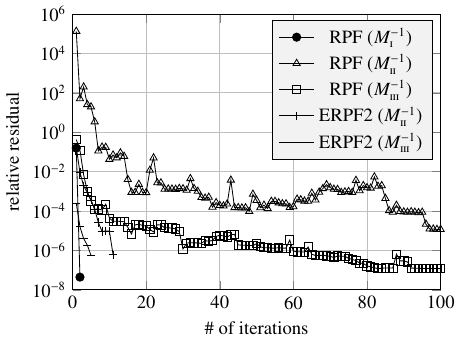}
    \caption{}
  \end{subfigure}
  \hfill\null 
	\caption{Test 1, Mandel's problem: Convergence profiles for $\Delta t = 10^{2}t_c$ (a) and $\Delta t = 10^{6}t_c$ 
	(b) with $a/h=10$.} 
  \label{fig:CPMan10}
\end{figure}

\begin{figure}
  \hfill
  \begin{subfigure}[t]{0.475\linewidth}
    \centering
    \includegraphics[scale=1]{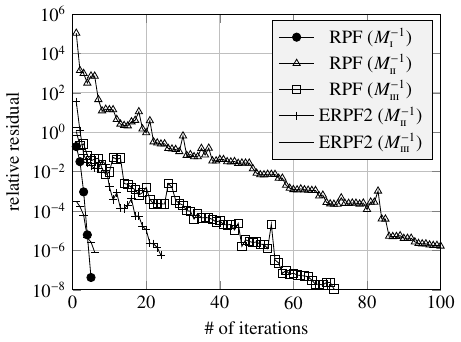}
    \caption{}
  \end{subfigure}
  \hfill
  \begin{subfigure}[t]{0.475\linewidth}
    \centering    
    \includegraphics[scale=1]{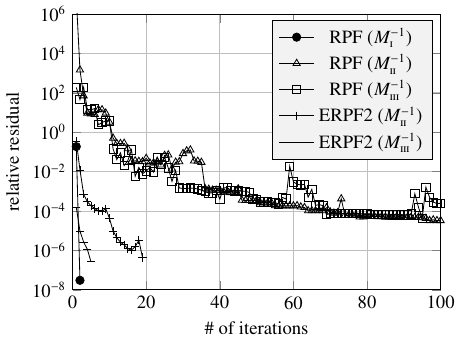}
    \caption{}
  \end{subfigure}
  \hfill\null 
	\caption{Test 1, Mandel's problem: the same as Figure \ref{fig:CPMan10} for $a/h=20$.}
  \label{fig:CPMan20}
\end{figure}

Of course, the results of Table \ref{tab:ManMet1} and \ref{tab:ManMet2} can change when the nested direct solvers
are replaced by inner preconditioners, such as with {the variants $\blkMat{M}^{-1}_{\textsc{ii}}$ and $\blkMat{M}^{-1}_{\textsc{iii}}$}. For example, Figure \ref{fig:CPMan10} 
and \ref{fig:CPMan20} compare the convergence profile for $a/h=10$ and $a/h=20$, respectively, and very large values 
of the time step ($\Delta t = 10^{2}t_c$ and $\Delta t = 10^{6}t_c$) obtained with the original RPF ({$\blkMat{M}^{-1}_{\textsc{i}}$, $\blkMat{M}^{-1}_{\textsc{ii}}$, 
and $\blkMat{M}^{-1}_{\textsc{iii}}$ variants}) and with EPRF2 ({$\blkMat{M}^{-1}_{\textsc{ii}}$ and $\blkMat{M}^{-1}_{\textsc{iii}}$ variants}). {All incomplete Cholesky factorizations
are computed in this case with zero fill-in, while algebraic multigrid is applied keeping the default parameters. It can be observed that, even in this academic example, the original RPF
preconditioner with an inexact solve for $\hat{K}$ or $\hat{A}$ might not converge with both $\blkMat{M}^{-1}_{\textsc{ii}}$ and $\blkMat{M}^{-1}_{\textsc{iii}}$ variants. Hence, this outcome appears to be an issue related to the nature of $\hat{K}$ and $\hat{A}$, rather than the choice of the inner preconditioner. By
distinction, the ERPF2 performance is just moderately affected by the use of inexact solves, exhibiting a very stable convergence profile with respect to a $\Delta t$ variation.} Hence, the
proposed algorithms are able to enhance not only the RPF performance, but also its robustness. 

The use of incomplete Cholesky factorizations with a partial fill-in as inner preconditioners of SPD blocks is efficient
in sequential computations, but can create concerns in parallel environments. Moreover, as it is well-known, it can
prevent from obtaining an optimal weak scalability with respect to the mesh size $h$. 
{For instance, this can be observed
in Figure \ref{fig:CPMan10} and \ref{fig:CPMan20}, ERPF2 profiles, where the number of iterations is not
constant with $h$ by using incomplete Cholesky factorizations.} The scalability
with $h$ can be restored by using scalable inner preconditioners, such as algebraic multigrid methods. 
The same analysis as Table \ref{tab:ManMet2} is here carried out by using the $\blkMat{M}_{\textsc{iii}}^{-1}$ variant.
Table \ref{tab:ManMet2_AMG} provides the outer iteration counts by varying $\Delta t$ and $h$.
As expected, the scalability with $h$ is much improved, 
since the number of iterations is quite stable between the progressive refinements.

\begin{table}
\caption{Test 1, Mandel's problem:  number of outer iterations with ERPF2 and $\blkMat{M}^{-1}_{\textsc{iii}}$ by varying the mesh-size
	$h$ and the time-step $\Delta t$.\label{tab:ManMet2_AMG}}
\small
\begin{center}
\begin{tabular}{lccccccc}
\toprule
	\multicolumn{1}{c}{$a/ h$} &  \multicolumn{1}{c}{$\Delta t / t_c$}&  \multicolumn{1}{c}{$\alpha/\alpha_K$} &\multicolumn{1}{c}{\# of iter.}& &  \multicolumn{1}{c}{$\Delta t / t_c$}&  \multicolumn{1}{c}{$\alpha/\alpha_A$} &\multicolumn{1}{c}{\# of iter.}\\
      \midrule
	   &$10^{-6}$&$1.3 \cdot 10^{-1}$&14 &&$10^{2}$&$3.5 \cdot 10^{-2}$&12 \\
	10 &$10^{-7}$&$4.3 \cdot 10^{-2}$&14 &&$10^{3}$&$1.1 \cdot 10^{-2}$&14 \\
	   &$10^{-8}$&$1.4 \cdot 10^{-2}$&14 &&$10^{4}$&$3.5 \cdot 10^{-3}$&14 \\
	   &&&&&&&\\   
	   &$10^{-6}$&$3.0 \cdot 10^{-1}$&19 &&$10^{2}$&$1.9 \cdot 10^{-2}$&13  \\
	20 &$10^{-7}$&$9.5 \cdot 10^{-2}$&19 &&$10^{3}$&$6.0 \cdot 10^{-3}$&12  \\
	   &$10^{-8}$&$3.0 \cdot 10^{-2}$&19 &&$10^{4}$&$1.9 \cdot 10^{-3}$&14  \\
	   &&&&&&&\\ 
	   &$10^{-6}$&$1.9 \cdot 10^{-1}$&22 &&$10^{2}$&$9.5 \cdot 10^{-3}$&11   \\
	40 &$10^{-7}$&$6.0 \cdot 10^{-2}$&19 &&$10^{3}$&$3.0 \cdot 10^{-3}$&12   \\
	   &$10^{-8}$&$1.8 \cdot 10^{-2}$&17 &&$10^{4}$&$9.5 \cdot 10^{-4}$&11    \\
	   &&&&&&&\\    
	   &$10^{-6}$&$7.9 \cdot 10^{-1}$&12 &&$10^{2}$&$6.2 \cdot 10^{-3}$&16 \\
	80 &$10^{-7}$&$2.5 \cdot 10^{-1}$&19 &&$10^{3}$&$1.9 \cdot 10^{-3}$&14 \\
	   &$10^{-8}$&$8.0 \cdot 10^{-2}$&19 &&$10^{4}$& $6.2 \cdot 10^{-4}$&11 \\
      \bottomrule
\end{tabular}
\end{center}
\end{table}

\subsection{Test2: Real-world applications}
The ERPF performance is finally tested in two large-size challenging cases, with unstructured grids and severe material anisotropy and
heterogeneity. 
{
In particular, the solver performance is investigated in presence of strong variability of both mechanical, i.e. compressibility, and hydraulic, i.e. permeability and porosity, material parameters, with severe jumps occurring between adjacent cells. This condition poses particular challenges as to the algorithmic robustness of the proposed methods.
}
We have considered two real-world applications:
\begin{itemize}
	\item Test 2a: Chaobai. This model is used to predict land subsidence due to a shallow multi-aquifer system 
		over-exploitation in the Chaobai River alluvial fan, North of Beijing, China \cite{FerGazCasTeaZhu17}. 
		The strong heterogeneity of the alluvial fan, which covers an overall areal extent of more than 1,100 km$^2$, 
		is accounted for by means of a statistical inverse framework in a multi-zone transition probability approach
		\cite{ZhuDaiGonGabTea16,ZhuFraGonFerDaiKeLiWanTea19}. Figure \ref{fig:chaobai_a} shows a reconstruction of the 
		lithofacies distribution. Details on the model 
		implementation are provided in \cite{FerGazCasTeaZhu17};
	\item Test 2b: SPE10. This is a typical petroleum reservoir engineering application reproducing a well-driven flow 
		in a deforming porous medium. The model setup is based on the 10\textsuperscript{th} SPE Comparative 
		Solution Project \cite{ChrBlu01}, a well-known severe benchmark in reservoir applications, assuming
		a poroelastic mechanical behavior with incompressible fluid and solid constituents. 
		{The porous medium is populated with homogeneous elastic properties, namely, Young's modulus $E = 8.3\cdot10^3$ MPa, Poisson's ratio $\nu =
0.3$, and Biot's coefficient $b = 1.0$.}
		The model is limited
		to the top 35 layers, which are representative of a shallow-marine Tarbert formation characterized by 
		extreme permeability variations (Figure \ref{fig:SPE10perm_b}), covering an areal extent of $365.76 \times 650.56$ 
		m\textsuperscript{2} and for $21.34$ m in the vertical direction. One injector and one production well, 
		located at opposite corners of the domain, penetrate vertically the entire reservoir and drive the
		porous fluid flow. The reader can refer to \cite{CheHuaMa06} for additional details.
\end{itemize}
The size and the number of non-zeros of the matrices arising from Test 2a and 2b are provided in Table \ref{tab:test2}.

\begin{figure}
  \centering
  \begin{subfigure}[t]{0.47\linewidth}
    \centering
    \includegraphics[width=\linewidth]{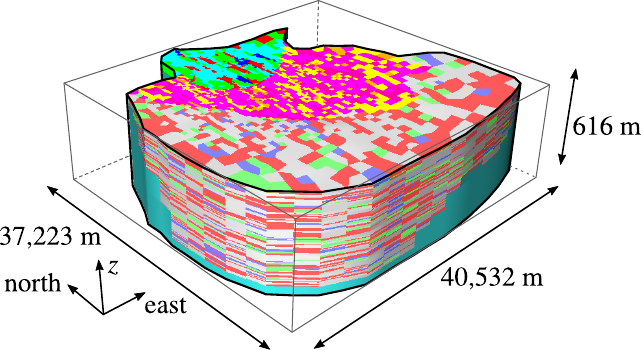}
    \caption{Test 2a, Chaobai: lithofacies heterogeneous distribution.}
    \label{fig:chaobai_a}
  \end{subfigure}
  \hfill
  \begin{subfigure}[t]{0.47\linewidth}
    \centering
    \includegraphics[width=\linewidth]{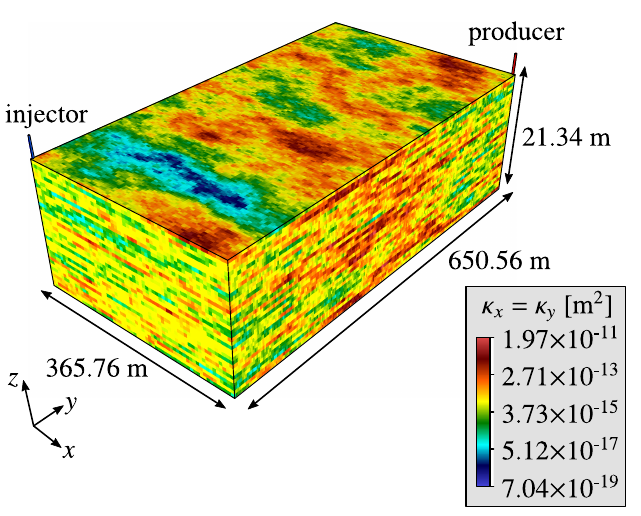}
    \caption{Test 2b, SPE10: horizontal permeability $\kappa_{x} = \kappa_{y} $.}
    \label{fig:SPE10perm_b}
  \end{subfigure}
  \caption{Test 2, Real-world applications: heterogeneous property distribution in the porous domain for the Chaobai (a) and SPE10 (b) test cases.}
  \label{fig:test2}
\end{figure}

\begin{table}
  \caption{Test 2, real-world applications: size and number of non-zeros of the test matrices.}\label{tab:test2}
  \small
  \begin{center}
    \begin{tabular}{ l l l l r }     
      \toprule
      \multicolumn{1}{c}{Test}&  \multicolumn{1}{c}{$n_u$}&  \multicolumn{1}{c}{$n_q$}& 
      \multicolumn{1}{c}{$n_p$}& \multicolumn{1}{c}{\# non-zeros} \\
      \midrule  
      2a: Chaobai  & 2,152,683 & 2,132,612 &   707,600 & 143,359,342 \\
      2b: SPE10    & 1,455,948 & 1,409,000 &   462,000 &  94,317,731 \\
      \bottomrule
    \end{tabular} 
  \end{center}
\end{table}

The size of the real-world problems addressed in Test 2 prevents from the use of the $\blkMat{M}_{\textsc{i}}^{-1}$ variant with
nested direct solvers. Therefore, we compare the performance of the original RPF with ERPF1 and ERPF2 in the
$\blkMat{M}_{\textsc{ii}}^{-1}$ variant, which makes use of incomplete Cholesky factorizations with partial fill-in as inner
preconditioners. We employ a limited-memory implementation \cite{LinMor99}, where the user can specify the row-wise
number of entries $\rho$ to be retained in addition to the row-wise number of non-zeros of the original matrix.
{The parameters of the inner preconditioners are set up to minimize the total CPU time $T_t$ with the RPF preconditioner and intermediate values of the time-step size $\Delta t$ ($10^{-4} < \Delta t < 10^{0}$ day for Test 2a and $10^{-6} < \Delta t < 10^{0}$ day  for Test 2b).
}

\begin{table}
\caption{Test 2a, Chaobai: iteration count and CPU times for a variable time-step size by using the $\blkMat{M}^{-1}_{\textsc{ii}}$ 
	variant as a preconditioner (Table \ref{tab:varPrec}). The fill-in parameter for: $\hat{K}_{\textsc{ic}}^{-1}$, 
	$\hat{K}_{\ell,\textsc{ic}}^{-1}$, and $K_{\textsc{ic}}^{-1}$ is $\rho_K=60$; $\hat{A}_{\textsc{ic}}^{-1}$ and $\hat{A}_{\ell,\textsc{ic}}^{-1}$ is 
	$\rho_A=50$; $\tilde{S}_{A,\textsc{ic}}^{-1}$ is $\rho_S=40$. The number of inner iteration for ERPF1 is 2.\label{tab:chaobaiMii}}
\small
\begin{center}
\begin{tabular}{cccccccccccccc}     
\toprule
	&&&\multicolumn{3}{c}{{ RPF}} &&\multicolumn{3}{c}{{ ERPF1}} && \multicolumn{3}{c}{{ ERPF2}}\\
\cmidrule{4-6} \cmidrule{8-10} \cmidrule{12-14}
	$\Delta t$ [day] &$\alpha/\alpha_K$&$\alpha/\alpha_A$&$n_{it}$ & $T_s$ [s] &$T_t$ [s] && $n_{it}$ & $T_s$ [s] &$T_t$ [s] &&$n_{it}$ & $T_s$ [s] &$T_t$ [s]\\  
\midrule
	$10^{-6}$ &$1.0 \cdot 10^{-2}$&$> 1$ &     49&  124.9 &  203.9 &&  25 &  102.2 &  181.4 &&  58 &216.0&286.6\\
	$10^{-5}$ &$2.8 \cdot 10^{-2}$&$> 1$ &     45&  114.2 &  192.4 &&  32 &  133.8 &  212.9 &&  54 &201.4&271.8\\
	$10^{-4}$ &$9.1 \cdot 10^{-2}$&$> 1$ &     45&  114.7 &  193.1 &&  29 &  118.2 &  197.2 &&  46 &170.1&240.1\\
	$10^{0}$ &$> 1$&$2.8 \cdot 10^{-1} $ &    247&  629.7 &  709.1 && 179 &  582.9 &  661.4 && 134 &274.4&334.0\\
	$10^{1}$ &$> 1$&$9.4 \cdot 10^{-2} $ &$>500$ &      - &      - && 273 &  884.0 &  962.4 && 127 &260.3&321.2\\
	$10^{2}$ &$> 1$&$2.8 \cdot 10^{-2} $ &$>500$ &      - &      - && 386 & 1247.5 & 1326.1 && 135 &277.4&337.2\\
\bottomrule
\end{tabular}
\end{center}
\end{table}

Table \ref{tab:chaobaiMii} provides the results obtained for Test 2a in terms of iteration count, solver CPU time $T_s$ 
and total CPU time $T_t$. For small $\Delta t$ values, i.e., when $\alpha<\alpha_K$, ERPF1 outperforms ERPF2, 
providing smaller iteration counts and total CPU times. In this test case, ERPF1 turns out to be comparable with the
original RPF, which generally proves, however, less robust. As $\Delta t$ increases, the iteration count with RPF
quickly grows and soon becomes unacceptable. With ERPF1, the iterations to converge increase as well, though at a 
lesser extent, while they remain stable when using ERPF2. In other words, inspection of Table \ref{tab:chaobaiMii}
reveals that, in a full transient problem where the time-step size typically increases as the simulation proceeds
towards the steady state, the most efficient strategy consists of switching from ERPF1 when $\alpha<\alpha_K$ to
ERPF2 when $\alpha<\alpha_A$, preserving the original RPF for the intermediate steps. 

\begin{table}
	\caption{Test 2b, SPE10: the same as Table \ref{tab:chaobaiMii} for RPF and ERPF2. The fill-in parameter for: 
	$\hat{K}_{\textsc{ic}}^{-1}$ is $\rho_K=30$; $\hat{A}_{\textsc{ic}}^{-1}$ is $\rho_A=40$; $\tilde{S}_{A,\textsc{ic}}^{-1}$ is $\rho_S=40$.
\label{tab:SPE10Mii}}
\small
\begin{center}
\begin{tabular}{lccccccccc}     
\toprule
	&&&\multicolumn{3}{c}{{ RPF}}&&\multicolumn{3}{c}{{ ERPF2}}\\
\cmidrule{4-6} \cmidrule{8-10} 
	$\Delta t$ [day] &$\alpha/\alpha_K$&$\alpha/\alpha_A$&$n_{it}$ & $T_s$ [s] &$T_t$ [s] && $n_{it}$& $T_s$ [s] &$T_t$ [s] \\  
 \midrule
	$10^{0}$ &$> 1$&$9.2 \cdot 10^{-4}$ &     348 & 467.7 & 503.3 && 84 & 79.8 & 102.2\\
	$10^{1}$ &$> 1$&$2.9 \cdot 10^{-4}$ & $>1000$ &     - &     - && 87 & 82.7 & 104.5\\
	$10^{2}$ &$> 1$&$9.2 \cdot 10^{-5}$ & $>1000$ &     - &     - && 47 & 44.8 &  67.3\\
\bottomrule
\end{tabular}
\end{center}
\end{table}

Table \ref{tab:SPE10Mii} provides the same results as Table \ref{tab:chaobaiMii} for Test 2b. In this case, the condition
$\alpha<\alpha_K$ is never met for time-step sizes with a relevant physical meaning. Therefore, we report only the
performance obtained by the original RPF and the ERPF2 methods, being ERPF1 more convenient for small $\Delta t$
only. As in Test 2a, ERPF2 always outperforms the original RPF preconditioner, proving much more robust and practically 
time-step independent.

\section{Conclusions}\label{sec:conc}
The Relaxed Physical Factorization introduced in \cite{FriCasFer19} has proved an efficient preconditioning strategy
for the three-field formulation of coupled poromechanics with respect to domain heterogeneity, anisotropy and grid
distortion. However, a performance degradation was observed 
for the values of time-step size that are typically required at the beginning of a full transient simulation and
toward steady-state conditions. In fact, at the beginning of the poromechanical process very small time steps are
necessary to capture accurately the pressure and deformation evolution in almost undrained conditions, while larger
and larger steps are convenient when proceeding towards the steady state. The origin of such a drawback stems from
the need of inverting, also inexactly, inner blocks in the form $\hat{C}=C+\beta F F^T$ with $\beta \gg 1$, i.e., 
where the rank-deficient term $FF^T$ prevails on the regular matrix $C$.

Two fully algebraic methods are presented to address the RPF issues: 
\begin{enumerate}
	\item a natural splitting of $\hat{C}$ is introduced to define a particular stationary scheme. The
		scheme is unconditionally convergent and does not require the inversion of $\hat{C}$. Hence, the inner solve
		with $\hat{C}$ is replaced by a few iterations of such a scheme;
	\item an appropriate projection operator is developed in order to annihilate the components of the
		solution vector of $\hat{C}\mybold{w}=\vec{b}$ lying in the near-null space of $\hat{C}$. The projected
		system is solved in a stable way, while the remaining part of the solution vector can be computed
		requiring the inversion of the regular SPD matric $C$ only, instead of $\hat{C}$.
\end{enumerate}
{The proposed methods can be generalized to applications where a similar algebraic issue arises, such as  in the use of an augmented Lagrangian approach for Navier-Stokes or incompressible elasticity. In the context of coupled poromechanics, they}
give rise to an Enhanced RPF preconditioner, ERPF1 and ERPF2, respectively, which has been
tested in both theoretical benchmarks and challenging real-world large-size applications. The following results are 
worth summarizing.
\begin{itemize}
	\item Both methods are effective in improving the performance and robustness of the original RPF algorithm
		in the most critical situations, i.e., $\Delta t \rightarrow 0$ and $\Delta t \rightarrow \infty$,
		stabilizing the iteration counts to convergence independently of the time step size.
	\item ERPF1 is usually more efficient for small time step values, while ERPF2 outperforms, also by a large
		amount, RPF and ERPF1 for large time step values. Therefore, the most convenient approach in a
		full transient poromechanical application appears to be switching from ERPF1 when $\alpha<\alpha_K$
		at the simulation beginning, to RPF when $\alpha>\alpha_K$ and $\alpha>\alpha_A$ with intermediate
		steps, to ERPF2 when $\alpha<\alpha_A$ approaching steady state conditions.
	\item The use of nested direct solvers ensures a scalable behavior of ERPF2 with respect to both the mesh
		and time spacings. Such a property is generally lost by using inexact inner solves, for instance by
		incomplete factorizations, but can be restored with scalable inner preconditioners, such as
		algebraic multigrid methods.
\end{itemize}

\section*{Acknowledgments}
Partial funding was provided by TotalEnergies S.A. through the FC-MAELSTROM Project.
Portions of this work were performed within the 2020 INdAM-GNCS project ``Optimization and advanced linear algebra for PDE-governed problems''.  Portions of this work were performed under the auspices of the U.S. Department of Energy by Lawrence Livermore National Laboratory under Contract DE-AC52-07NA27344.
The authors are grateful to two anonymous reviewers, whose comments significantly helped improve the presentation.

\appendix

\section{Alternative ERPF2 application}\label{app:altERPF2}
With the ERPF2 approach, the inner solution to $\hat{C}\mybold{w}=\vec{b}$ is obtained through equation \eqref{eq:precSMW}
whenever $\beta>\beta_\ell$. With specific reference to Algorithm \ref{alg:appRPF}, this can happen twice at each
preconditioner application either:
\begin{enumerate}
	\item for $\alpha<\alpha_K$, having set $C=K$, $F=Q$, and $\beta=\alpha^{-1}$, or 
	\item for $\alpha<\alpha_A$, having set $C=A$, $F=B$, and $\beta=\gamma/\alpha$.
\end{enumerate}
For the sake of brevity, let us consider the case no.~2, which typically occurs more often in practical applications,
e.g., for large values of the time-step size. Similar considerations hold for the case no.~1.

\begin{algorithm}[t]
  \caption{Alternative ERPF2 application: $[\mybold{t}_u^T,\mybold{t}_q^T,\mybold{t}_p^T]=\mbox{\sc AltERPF2\_Apply}(\mybold{r}_u,\mybold{r}_q,\mybold{r}_p,\blkMat{A},\blkMat{M})$}
  \label{alg:appRPFmet2}
  \begin{algorithmic}[1]
    \If{$\alpha > \alpha_K$}
      \State $\mybold{x}_u=\alpha^{-1} Q \mybold{r}_p$
      \State $\mybold{x}_u \gets \mybold{r}_u + \mybold{x}_u$
      \State \textbf{Apply} $M^{-1}_{\hat{K}}$ to $\mybold{x}_u$ to get $\mybold{t}_u$   
      \State $\mybold{y}_p=Q^T \mybold{t}_u$
      \State  $\mybold{y}_p \gets \mybold{r}_p - \mybold{y}_p$
    \Else
      \State \textbf{Apply} $M^{-1}_{K}$ to $\mybold{r}_u$ to get $\mybold{x}_u$ 
      \State $\mybold{x}_p=Q^T \mybold{x}_u$
      \State  $\mybold{x}_p \gets \mybold{r}_p - \mybold{x}_p$
      \State \textbf{Apply} $M^{-1}_{\tilde{S}_K}$ to $\mybold{x}_p$ to get $\mybold{y}_p$ 
      \State  $\mybold{x}_u=Q \mybold{y}_p$
      \State  $\mybold{x}_u \gets \alpha \mybold{r}_u + \mybold{x}_u$
      \State \textbf{Apply} $M^{-1}_{K}$ to $\mybold{x}_u$ to get $\mybold{t}_u$ 
    \EndIf
    \If{$\alpha > \alpha_A$}
      \State $\mybold{z}_q=B \mybold{y}_p$
      \State $\mybold{z}_q \gets \mybold{r}_q + \alpha^{-1} \mybold{z}_q$      
      \State \textbf{Apply} $M^{-1}_{\hat{A}}$ to $\mybold{z}_q$ to get $\mybold{t}_q$   
      \State $\mybold{t}_p=B^T \mybold{t}_q$
      \State $\mybold{t}_p \gets \alpha^{-1} \left( \mybold{y}_p - \gamma \mybold{t}_p \right)$     
    \Else
      \State \textbf{Apply} $M^{-1}_{A}$ to $\mybold{r}_q$ to get $\mybold{z}_q$ 
      \State $\mybold{z}_p=B^T \mybold{z}_q$
      \State  $\mybold{y}_p \gets \alpha^{-1}\left ( \mybold{y}_p -\gamma \mybold{z}_p \right )$
      \State \textbf{Apply} $M^{-1}_{\tilde{S}_A}$ to $\mybold{y}_p$ to get $\mybold{t}_p$ 
      \State  $\mybold{z}_q=B \mybold{t}_p$
      \State  $\mybold{z}_q \gets \mybold{r}_q + \mybold{z}_q$
      \State \textbf{Apply} $M^{-1}_{A}$ to $\mybold{z}_q$ to get $\mybold{t}_q$ 
    \EndIf
  \end{algorithmic}
\end{algorithm}

Consider Algorithm \ref{alg:appRPF}. Recalling that $\mybold{z}_q= \mybold{r}_q + \alpha^{-1} B \mybold{y}_p$ (row 7), 
the application of $\hat{A}^{-1}$ on $\mybold{z}_q$ to get $ \mybold{t}_q$ using equation \eqref{eq:precSMW} yields:
\begin{eqnarray}
	\mybold{t}_q &=&\hat{A}^{-1} \mybold{z}_q \nonumber \\
	&=& A^{-1} \mybold{r}_q + \alpha^{-1} A^{-1} B \mybold{y}_p - \beta A^{-1} B \tilde{S}_A^{-1} B^T A^{-1} 
	\mybold{r}_q - \beta \alpha^{-1} A^{-1} B \tilde{S}_A^{-1} S_A \mybold{y}_p \nonumber \\
	&=& A^{-1} \mybold{r}_q + \alpha^{-1} A^{-1} B \tilde{S}_A^{-1} \left(  \tilde{S}_A - \beta S_A \right) 
	\mybold{y}_p -\beta A^{-1} B \tilde{S}_A^{-1} B^T A^{-1} \mybold{r}_q \nonumber \\
	&=& A^{-1} \mybold{r}_q + \alpha^{-1} A^{-1} B \tilde{S}_A^{-1} \mybold{y}_p - \beta A^{-1} B \tilde{S}_A^{-1} 
	B^T A^{-1} \mybold{r}_q \nonumber \\
	&=& A^{-1} \mybold{r}_q + A^{-1} B \tilde{S}_A^{-1} \left( \alpha^{-1} \mybold{y}_p - \beta B^T A^{-1} 
	\mybold{r}_q \right ) \nonumber \\
	&=& A^{-1} \left[ \mybold{r}_q + B \tilde{S}_A^{-1} \left( \alpha^{-1} \mybold{y}_p - \beta B^T A^{-1} \mybold{r}_q
	\right) \right]. \label{eq:equivalentq} 
\end{eqnarray}
Recalling that $\beta=\gamma/\alpha$, lines 9 and 10 of Algorithm \ref{alg:appRPF} can be rearranged as follows:
\begin{eqnarray}
	\mybold{t}_p &=& \alpha^{-1} \mybold{y}_p - \beta B^T \mybold{t}_q \nonumber \\
	&=& \alpha^{-1} \mybold{y}_p - \beta B^T A^{-1} \mybold{r}_q - \beta B^T A^{-1} B \tilde{S}_A^{-1} \left( 
	\alpha^{-1} \mybold{y}_p - \beta B^T A^{-1} \mybold{r}_q  \right) \nonumber \\ 
	&=& \left( I_p - \beta S_A \tilde{S}_A^{-1} \right) \left( \alpha^{-1} \mybold{y}_p - \beta B^T A^{-1} 
	\mybold{r}_q \right) \nonumber \\
	&=& \tilde{S}_A^{-1} \left( \alpha^{-1} \mybold{y}_p - \beta B^T A^{-1} \mybold{r}_q \right). \label{eq:equivalentp}
\end{eqnarray}
The vectors $\mybold{t}_p$ and $\mybold{t}_q$ are therefore computed by:
\begin{eqnarray}
	\mybold{t}_p &=& \tilde{S}_A^{-1} \left( \alpha^{-1} \mybold{y}_p - \beta B^T A^{-1} \mybold{r}_q  \right), 
	\label{eq:equivalentFact1} \\
	\mybold{t}_q &=& A^{-1} \left( \mybold{r}_q + B \mybold{t}_p \right) \label{eq:equivalentFact2}.
\end{eqnarray}
Equations \eqref{eq:equivalentFact1} and \eqref{eq:equivalentFact2} shows that Method 2 used for the inner solve with
$\hat{A}$ is equivalent to apply the inverse of the following block factorization for $\blkMat{M}_2$:
 \begin{equation}\label{eq:RPF_fact}
  \begin{bmatrix} I_u & 0 & 0\\ 0 & A & -B\\ 
  0 & \gamma B^T & \alpha I_p  \\ \end{bmatrix} =
    \begin{bmatrix}  I_u & 0 & 0\\ 0 & A & 0\\ 
  0 & \gamma B^T & \tilde{S}_A \\ \end{bmatrix}  
  \begin{bmatrix}  I_u & 0 & 0\\ 0 & I_q & -A^{-1} B\\ 
  0 & 0 & \alpha I_p \end{bmatrix}.
\end{equation} 
In fact, solving the block system $\blkMat{M}_2\vec{t}=\vec{y}$ with the aid of \eqref{eq:RPF_fact} yields:
\begin{equation} \label{eq:M2app}
	\left( \begin{array}{c}
		\vec{t}_u \\ \vec{t}_q \\ \vec{t}_p
	\end{array} \right) =
	\left[ \begin{array}{ccc}
		I_u & 0 & 0 \\ 0 & I_q & \alpha^{-1} A^{-1} B \\ 0 & 0 & \alpha^{-1} I_p 
	\end{array} \right] \left[ \begin{array}{ccc}
		I_u & 0 & 0 \\ 0 & A^{-1} & 0 \\ 0 & -\gamma \tilde{S}_A^{-1} B^T A^{-1} & \tilde{S}_A^{-1} 
	\end{array} \right] \left( \begin{array}{c} 
		\vec{y}_u \\ \vec{y}_q \\ \vec{y}_p 
	\end{array} \right), 
\end{equation}
which provides equations \eqref{eq:equivalentFact1}-\eqref{eq:equivalentFact2} for $\vec{y}_q=\vec{r}_q$.

With similar arguments, after some algebra it can be easily proved that Method 2 for the inner solve with $\hat{K}$
is equivalent to apply the inverse of the following factorization of $\blkMat{M}_1$:
\begin{equation}\label{eq:M1_fact}
  \begin{bmatrix} K & 0 & -Q\\ 0 & \alpha I_q & 0\\ Q^T & 0 & \alpha I_p  \\ \end{bmatrix} =
    \begin{bmatrix}  K & 0 & 0\\ 0 & \alpha I_q & 0\\ Q^T & 0 & \tilde{S}_K \\ \end{bmatrix}  
	    \begin{bmatrix}  I_u & 0 & -K^{-1}Q \\ 0 & I_q & 0\\ 0 & 0 & \alpha I_p \end{bmatrix}.
\end{equation}
The solution to the system $\alpha \blkMat{M}_1 \vec{y} = \vec{r}$ using \eqref{eq:M1_fact} reads:
\begin{equation} \label{eq:aM1app}
        \left( \begin{array}{c}
                \vec{y}_u \\ \vec{y}_q \\ \vec{y}_p
        \end{array} \right) = \alpha
        \left[ \begin{array}{ccc}
		I_u & 0 & \alpha^{-1} K^{-1} Q \\ 0 & I_q & 0 \\ 0 & 0 & \alpha^{-1} I_p 
        \end{array} \right] \left[ \begin{array}{ccc}
		K^{-1} & 0 & 0 \\ 0 & \alpha^{-1} I_q & 0 \\ -\tilde{S}_K^{-1} Q^T K^{-1} & 0 & \tilde{S}_K^{-1} 
        \end{array} \right] \left( \begin{array}{c} 
                \vec{r}_u \\ \vec{r}_q \\ \vec{r}_p 
        \end{array} \right), 
\end{equation}
which provides:
\begin{eqnarray}
	\vec{y}_p & = & \tilde{S}_K^{-1} \left( \vec{r}_p - Q^T K^{-1} \vec{r}_u \right), \label{eq:yp} \\
	\vec{y}_u & = & K^{-1} \left( \alpha \vec{r}_u + Q \vec{y}_p \right), \label{eq:yu}
\end{eqnarray}
with $\vec{y}_q=\vec{r}_q$.

As already noticed in Section \eqref{subsec:met2}, the inner solves with $K$, $A$, $\tilde{S}_K$, and $\tilde{S}_A$ 
can be conveniently replaced by inexact solves with the aid of inner preconditioners, say $M_K^{-1}$, $M_A^{-1}$,
$M_{\tilde{S}_K}^{-1}$, and $M_{\tilde{S}_A}^{-1}$, respectively. The equivalent ERPF2 application obtained by
using the factorizations above is finally summarized in Algorithm \ref{alg:appRPFmet2}.

\end{document}

%% file: packagesAndMacros.tex
\usepackage[utf8]{inputenc}

\usepackage{amssymb,amsmath,amsfonts,mathtools}
\usepackage{amsthm}
	\theoremstyle{remark}
	\newtheorem{remark}{Remark}[section]

	\theoremstyle{plain}
	\newtheorem{theorem}{Theorem}[section]
	
	\newtheorem{lemma}[theorem]{Lemma}
	\newtheorem{defin}[theorem]{Definition}

\usepackage{graphicx}
\usepackage{graphbox}
\usepackage{algorithm}
\usepackage{algpseudocode}
\usepackage{subcaption}
\usepackage{booktabs,cellspace}
\usepackage{multirow}
\usepackage{makecell}
\usepackage{ulem}

\usepackage{tikz}
\usepackage{pgfplots}
\usetikzlibrary{patterns}
\pgfplotsset{compat=newest}
\usetikzlibrary{calc}

\usepackage{natbib}
\biboptions{sort&compress}
\usepackage{lineno}
\usepackage[colorlinks=true, allcolors=blue]{hyperref}
\usepackage{hypernat}
\modulolinenumbers[5]
\usepackage[]{setspace}

\usepackage[colorinlistoftodos,prependcaption,textsize=scriptsize]{todonotes}

\usepackage{listings}
\usepackage{color} 
\definecolor{mygreen}{RGB}{28,172,0} 
\definecolor{mylilas}{RGB}{170,55,241}

\renewcommand{\vec}[1]{\mathbf{#1}}
\newcommand{\mybold}[1]{\mathbf{#1}}

\newcommand{\tensorOne}[1]{\boldsymbol{#1}}
\newcommand{\tensorTwo}[1]{\boldsymbol{#1}}
\newcommand{\tensorFour}[1]{\mathbb{#1}}

\renewcommand{\Vec}[1]{\boldsymbol{\mathbf{#1}}}
\newcommand{\blkMat}[1]{\boldsymbol{\mathbf{#1}}}

\newcommand{\ran}[1]{\operatorname{Ran}( #1 )}

